\documentclass[french,12pt,openany,letterpaper]{article}
\usepackage{latexsym,amssymb}
\usepackage{amsthm}
\usepackage{amsmath}
\usepackage{float}
\usepackage{graphics}
\usepackage{graphicx}
\usepackage{epstopdf}
\usepackage[centerlast,sc,small]{subfigure}
\usepackage{flafter}
\usepackage{endnotes}
\usepackage{mathrsfs}
\usepackage{caption}
\usepackage{xy}
\usepackage{algorithm,algorithmic}
\usepackage{placeins}
\usepackage{bpchem}

\xyoption{all} \xyoption{matrix}

\newtheorem{theorem}{Theorem}[section]

\newtheorem{corollary}{Corollary}[section]

\newtheorem{lemma}{Lemma}[section]
\newtheorem{proposition}{Proposition}[section]
\newtheorem{remark}{Remark}[section]

\newcommand{\Huz}{H^{1}_{0}\left(\Omega\right)}

\newcommand{\Hd}{H^{2}\left(\Omega\right)}

\newcommand{\keywords}[1]{\small\textbf{Keywords: }#1}
\newcommand{\AMSsubj}[1]{\textbf{AMS subject classifications: }#1}     {\endlist}

\title{Bilateral obstacle optimal control Problem}

\author{\textsc{R. Ghanem$^{\dag \ddag}$, B. Zireg$^\ddag$}  \\
{\footnotesize  {$^{\dag \ddag}$ Numerical analysis, optimization and statistical laboratory (LANOS)}}\\
{\footnotesize {Badji-Mokhtar, Annaba University}}\\
{\footnotesize P.O. Box 12, 23000, Annaba Algeria}\\
{\footnotesize{ $^\dag $ radouen.ghanem@univ-annaba.org }}}

\date{}

\begin{document}
\maketitle

 \begin{abstract}
In this work we consider the numerical resolution of the bilateral
obstacle optimal control problem given in Bergounioux et al \cite{BergLenhBi}. Where
the main feature of this problem
is that the control and the obstacle are the same.%
\end{abstract}

\keywords Optimal control, obstacle problem, finite differences
method.%

\AMSsubj    65K15   49K10, 35J87 , 49J20 , 65L12 , 49M15.%

\section{ Introduction}
 Variational inequalities and related
optimal control problems have been recognized as suitable
mathematical models for dealing with many problems arising in
different fields, such as shape optimization theory, image
processing and mechanics, (see for example \cite{BergouShape},
\cite{BergouImage}, \cite{DesongMech}, \cite{lionstamp},
\cite{Rodri}).

Optimal control problem governed by variational inequalities has
been studied extensively during the last years by many authors,
such as \cite{Bar1}, \cite{MignPuel}, \cite{Mign}. These authors
have studied optimal control problems for obstacle problems (or
variational inequalities) where the obstacle is a known function,
and the control variables appear as variational inequalities. In
other words, controls do not change the obstacle and, on the other
hand , in \cite{AdamsLenhart}, \cite{BergLenhBi},
\cite{Ghanemlivre} the authors have studied another class of
problems where the obstacle functions are unknowns and are
considered as control functions.

In this paper, we investigate optimal control problems governed by
variational inequalities of obstacle type. This kind of problem is
very important and it can lead to the shape optimization problem
governed by variational inequality, it may concern the optimal shape of dam \cite{Chipotlivre}, for which the obstacle
gives the shape to be designed such that the pressure of the fluid
inside the dam is close to a desired value. Besides, if we want to
design a membrane having an expected shape, we need to choose a
suitable obstacle. In this case, the obstacle can be considered as
a control, and the membrane as the state (see for example
\cite{HlavacekBeam}).

It should be pointed out that, in the optimal control problem of a
variational
inequality, the main difficulty comes from the fact that the mapping $%
\mathcal{T}$  between the control and the state (control-to-state
operator) is not Gateaux differentiable as pointed it out in
\cite{Mign},  \cite{MignPuel} where one can only define a conical
derivative for $\mathcal{T}$  but only Lipschitz-continuous and so
it is not easy to get optimality conditions that can be
numerically exploitable.

To overcome this difficulty, different authors (see for example,
Kunisch et al.\cite{Kunish}  V. Barbu \cite{Bar1} and the
references therein) consider a Moreau-Yosida approximation
technique to reformulate the governing variational inequality
problem into a problem governed by a variational equation. Our
approach is based on the penalty method and Barbu's treatment as a
penalty parameter approaching zero. We then obtain a system of
optimality for suitable approximations of the original problem
which can be easily used from the numerical point of view.

Nevertheless, the optimal control of variational inequalities of
obstacle type is still a very active field of research especially
for their numerical treatment which are given in the recent
publication \cite{Ghanemzireg}.

The problem that we are going to study can be set in a wider class
of problems, which can be formally described as follows

\begin{equation*}\label{p0}
\min \left\{ J\left( y,\chi \right) ,y=\mathcal{T}\left( \chi
\right) ,\chi \in \mathcal{U}_{ad}\subset \mathcal{U}\right\}
\end{equation*}

where $\mathcal{T}$ is an operator which associates $y$ to $\chi
$, when $y$ is a solution to

\begin{equation}\label{obs}
\forall y\in \mathcal{K}\left( y,\chi \right) ,\left\langle
A\left( y,\chi \right) ,y-v\right\rangle \geq 0 \tag{$obs$},
\end{equation}

where $K$ is a multiplication from $\chi \times \mathcal{U}$ to
$2^{\chi }$ when $\chi $ is a Banach space and $A$ is a
differential operator from $Y$ to the dual $Y^{\prime }$. Let $h$
be an application from $\mathbb{R}\times \mathbb{R}$ to
$\mathbb{R}$, then the variational inequality that relates the
control $\chi $ to the state $y$ can be written as

\begin{equation*}
\left\langle A\left( y,\chi \right) ,y-v\right\rangle
_{Y,Y^{\prime }}+h\left( \chi ,v\right) -h\left( \chi ,y\right)
\geq \left( \chi ,v-y\right) ,\forall y\in \mathcal{Y},
\end{equation*}

where this formulation gives the obstacle problems where the
obstacle is the control.

Following the previous ideas, we may apply a smoothed penalization
approach to our problem. More precisely, the idea is to
approximates the obstacle problem
 by introducing an approximating parameter $\delta
$, where the approximating method is based on the penalization
method and it consists in replacing the obstacle problem
(\eqref{obs} by a family of semilinear equations. In
\cite{BergLenhBi}, Bergounioux et al. considered the following
bilateral optimal control obstacle problem

\begin{multline}
\min \{ J\left( \varphi ,\psi \right) =\frac{1}{2}\int\nolimits_{%
\Omega }\left( \mathcal{T}\left( \varphi ,\psi \right) -z\right) ^{2}dx+%
\frac{\nu }{2}\int\nolimits_{\Omega }\left( \left( \Delta \varphi
\right) ^{2}+\left( \Delta \psi \right) ^{2}\right) dx,\\
 \left(
\varphi ,\psi \right) \in \mathcal{U}_{ad}\times
\mathcal{U}_{ad}\}
\end{multline}

where $\nu$ is a given positive constant and $z$ belongs to
$L^{2}\left( \Omega \right) $ as a target profile,

such that $y=\mathcal{T}\left( \varphi \right)$ is a solution of
the bilateral obstacle problem given by

\begin{equation*}
\left\langle Ay,v-y\right\rangle \geq \left( f,v-y\right) ,\text{
for all } v\text{ in } \mathcal{K}\left( \varphi ,\psi \right),
\end{equation*}

where $\mathcal{K}\left( \varphi,\psi \right)$ is given by

\begin{equation*}
\mathcal{K}\left( \varphi,\psi \right) =\left\{ y\in
H_{0}^{1}\left( \Omega \right) ,\psi \geq y\geq \varphi \right\},
\end{equation*}

and the set of admissible controls $\mathcal{U}_{ad}$ is defined
as follows

\begin{equation*}
\mathcal{U}_{ad}=\left\{ \left( \varphi ,\psi \right) \in
\mathcal{U} \times \mathcal{U} \mid \varphi \leq \psi \right\},
\end{equation*}

where $\mathcal{U}= H^{2}\left( \Omega \right) \times
H_{0}^{1}\left( \Omega \right) $. As we need $H^{2}-$priori
estimate, \ we could assume that $\mathcal{U}_{ad}$ is $H^{2} $
bounded. For example, we can suppose that $\mathcal{U}_{ad}$ \ is
$\mathcal{B}_{H^{2}}\left( 0,R\right)$ i.e. a ball of center $0$
and radius $R$, where $R$ is a large enough positive real number,
but according to \cite{Ghanemzireg}, this choice can lead to
technical difficulties to get a numerical solution of the
optimality system.

In \cite{Ghanemzireg}, Ghanem et al., have solved numerically the
unilateral optimal control of obstacle problem given by

\begin{equation}
\label{cont_pb_lapla} \min \left\{ J\left( \varphi \right)
=\frac{1}{2}\int\nolimits_{\Omega
}\left( \mathcal{T}\left( \varphi \right) -z\right) ^{2}dx+\frac{\nu }{2}%
\int\nolimits_{\Omega }\left( \Delta \varphi \right)
^{2}dx,\varphi \in \mathcal{U}\right\}
\end{equation}

instead of the one defined by

\begin{equation*}
\min \left\{ J\left( \varphi \right)
=\frac{1}{2}\int\nolimits_{\Omega
}\left( \mathcal{T}\left( \varphi \right) -z\right) ^{2}dx+\frac{\nu }{2}%
\int\nolimits_{\Omega }\left( \nabla \varphi \right)
^{2}dx,\varphi \in \mathcal{U}_{ad}\right\}
\end{equation*}

where

\begin{equation}
\label{Radius_R} \mathcal{U}_{ad}=\left\{ \varphi \in H^{2}\left(
\Omega \right) ,\varphi \in \mathcal{B}_{H^{2}}\left( 0,R\right)
\right\}
\end{equation}

such that $y=\mathcal{T}\left( \varphi \right) $ is a solution of
the unilateral obstacle problem given by

\begin{equation*}
\left\langle Ay,v-y\right\rangle \geq \left( f,v-y\right) ,\text{
for all } v\text{ in }\mathcal{K}\left( \varphi \right)
\end{equation*}

where $\mathcal{K}\left( \varphi \right) $ is defined by

\begin{equation*}
\mathcal{K}\left( \varphi \right) =\left\{ y\in H_{0}^{1}\left(
\Omega \right) ,y\geq \varphi \right\}.
\end{equation*}

According to the result given in \cite{Ghanem} the authors point
out that, in spite of the elimination of the inequality constraint
given by \eqref{Radius_R}, we still get a local convergence
property implied by the constraint $\left\Vert \varphi
_{n}\right\Vert _{H^{2}\left( \Omega \right) }\leq R$. Hence, we
are again confronted to the inequality constraint
\eqref{Radius_R}.

So we note that it is not necessary to suppress the constraint
\eqref{Radius_R}, because it is going to appear again to get the
local convergence of the algorithm used for the numerical solution
of the problem give by \eqref{cont_pb_lapla}.

For the numerical solution of optimal control problem, it is usual to use two kinds of numerical approaches:  direct
and indirect methods.  Direct methods consist in discretizing  the cost function, the state and
the control and thus reduce the problem to a nonlinear optimization problem with constraints.
Indirect methods consist of solving numerically the optimality system given by the state, the adjoint and the projection equations.

The aim of this paper is the numerical solution of the optimal
control problem given in \cite{BergLenhBi} by using the indirect approach (after optimisation)   based on the same idea
and techniques given in \cite{Ghanemzireg}, where the optimality
system is characterized by

\begin{equation*}
\left\{
\begin{array}{l}
Ay^{\delta }+(\beta_{\delta } (y^{\delta }-\varphi ^{\delta
})-\beta_{\delta } (\psi ^{\delta }-y^{\delta }))=f\text{ in }\Omega \text{ and } y^{\delta}=0 \text{ on }\partial\Omega  \\
\\
A^{\ast }p^{\delta }+\mu _{1}^{\delta }+\mu _{2}^{\delta }=y^{\delta }-z%
\text{ in }\Omega \text{ and } p^{\delta}=0 \text{ on }\partial\Omega   \\
\\
\left(\mu_{1}+\varphi^{\delta}-\varphi^{\ast}, \varphi -
\varphi^{\delta}\right)+ \left(\mu_{2}+\psi^{\delta}-\psi^{\ast},
\psi - \psi^{\delta}\right)+ \\ + \nu \left(\Delta
\varphi^{\delta}, \Delta\left(\varphi -
\varphi^{\delta}\right)\right)+\nu \left(\Delta \psi^{\delta},
\Delta\left(\psi - \psi^{\delta}\right)\right)=0,
 \text{ for all } \varphi \text{ in } \mathcal{U}_{ad}
\end{array}%
\right.
\end{equation*}

For the numerical solution, we first begin by discretizing the
optimality system by using finite differences schemes and then by
proposing an iterative algorithm based on Gauss-Seidel method that
is a combination of damped-Newton-Raphson and a direct method.

The main difficulties of this work compared to the one considered
in \cite{Ghanemzireg}, is to get an optimality system numerically
exploitable by the proposed algorithm.

In the sequel, we denote by $\mathcal{B}_{V}\left( 0,r\right) $
the $V$-ball around $o$ of radius $r$ and by $C$ generic positive
constants.

The rest of paper is organized as follows: in section 2 we give
precise assumptions and some well-known results. In section 3, we
introduce the iterative algorithm and give convergence results to
solve the optimality system. Section 4 is devoted to numerical
examples that illustrate the theoretical findings and in section 5
we present some remarks and a conclusion.

\section{Preliminaries and known results}

We consider the bilinear form $\sigma (\cdot ,\cdot )$ defined in $%
H^{1}(\Omega )\times H^{1}(\Omega )$, where we assume that the
following conditions are fulfilled

\begin{description}

\item[H\textsubscript{1}.] Continuity%
\begin{equation*}
\exists \,C>0,\forall u ,v \in H^{1}(\Omega ),\,\left\vert \sigma
(u ,v )\right\vert \leq C\left\Vert u \right\Vert _{H^{1}(\Omega
)}\left\Vert v \right\Vert _{H^{1}(\Omega )} \label{conbilfora}
\end{equation*}%
\end{description}
\begin{description}
\item [H\textsubscript{2}.] Coercivity%
\begin{equation*}
\exists \,c>0,\forall u \in H^{1}(\Omega ),\,\sigma (u ,u )\geq
c\left\Vert u \right\Vert _{H^{1}(\Omega )}^{2}
\label{coebiylfora}
\end{equation*}
\end{description}

We call $A$ in $ \mathcal{L}(H^{1}(\Omega ),H^{-1}(\Omega ))$ the
linear self-adjoint
elliptic operator (see \cite{LionMage}) associated to $\sigma $ such that $%
\left\langle Au,v\right\rangle =\sigma (u,v)$, and assume that the
adjoint form $\sigma^{*} (\cdot,\cdot)$ satisfies the conditions
H\textsubscript{1} and H\textsubscript{2}.

For any $\varphi$ and $\psi $ in $ H_{0}^{1}\left( \Omega \right)
$, we define

{\small
\begin{equation}
\mathcal{K}(\varphi ,\psi )=\left\{ y\in H_{0}^{1}(\Omega )\mid
\psi \geq y\geq \varphi ~\text{\ in }~\Omega \right\},  \label{k}
\end{equation}%
}

and consider the following variational inequality

{\small
\begin{equation}
\sigma \left( y,\,v-y\right) \geq \left( f,\,v-y\right) ,\quad
\text{ for all }v\text{ in }\mathcal{K}(\varphi ,\psi ),
\label{eq:var_eq}
\end{equation}%
}

where $f$ belongs to $L^{2}\left( \Omega \right) $ is a source
term. From now on, we define the operator $\mathcal{T}$
(control-to-state operator) from $\mathcal{U}\times \mathcal{U}$ to $\mathcal{U}$, such that $%
y=\mathcal{T}\left( \varphi ,\psi \right) $ is the unique solution
to the obstacle problem given by $\left( \ref{k}\right) $ and
$\left( \ref{eq:var_eq}\right) $
(see \cite{lionstamp}), where $\mathcal{U}=H^{2}(\Omega )\cap H_{0}^{1}(\Omega )$.\\
 Let $\mathcal{U}_{ad}$ be the
set of admissible controls which is assumed to be $
H^{2}(\Omega)$-bounded subset of $H^{2}(\Omega)\cap
H_{0}^{1}(\Omega)$, convex and closed in $H^{2}(\Omega)$. We may
choose, for example,
\begin{equation}
 \mathcal{U}_{ad} =\mathcal{B}_{H^{2}}(0,R)=\{ v \text{ in }H^{2}(\Omega)\cap
H_{0}^{1}(\Omega)|\|v\|_{H^{2}}\leq R\}
 \end{equation}
 where $R$ is a large enough positive real number. This boundedness
assumption for $\mathcal{U}_{ad}$ is crucial: it gives a priori
$H^{2}$ - estimates on the control functions and leads to the
existence of a solution. Now, we consider the optimal control
problem \eqref{cstfunctn} defined as follows

\begin{multline}\label{cstfunctn}
\min \{ J(\varphi ,\psi )=\tfrac{1}{2}\int\nolimits_{\Omega
}\left(
\mathcal{T}\left( \varphi ,\psi \right) -z\right) ^{2}d{x}+\tfrac{\nu }{2}%
\left( \int\nolimits_{\Omega }\left( \left( \nabla \varphi \right)
^{2}+\left( \nabla \psi \right) ^{2}\right) d{x}\right) ,\\
\text{ \ for all \ }\varphi ,\psi \in \mathcal{U}_{ad}\} ,%
\tag{P}
\end{multline}%

where $\nu $ is a strictly given positive constant, $z$ in
$L^{2}\left( \Omega \right) $. We
seek the obstacles (optimals controls) $\left( {\small \bar{\varphi},\bar{\psi}}%
\right) $ in $\mathcal{U}_{ad}^{2}$, such that the corresponding
state is close to a target profile $z$.

To derive necessary conditions for an optimal control, we would
like to differentiate the map $\left( {\small \varphi ,\psi
}\right) \mapsto \mathcal{T}\left( {\small \varphi ,\psi }\right)
$. Since the map $\left( {\small \varphi ,\psi }\right) \mapsto
\mathcal{T}\left( {\small \varphi ,\psi }\right) $ is not
directly differentiable (see \cite{Mign}), the idea here consists in approximating the map $%
\mathcal{T}\left( \varphi ,\psi \right) $ by a family of maps
$\mathcal{T}^{\delta }\left(
\varphi ,\psi \right) $ and replacing the obstacle problem $\left( \ref%
{eq:var_eq}\right) $ and \eqref{k} by the following smooth
semilinear equation (see \cite{MignPuel}, \cite{Brezis2}):

\begin{equation*}
Ay+\left( \beta _{\delta }\left( y-\varphi \right) -\beta _{\delta
}\left(
\psi -y\right) \right) =f\quad \text{in }~\Omega ,\text{ and } y=0~\text{ on }%
~\partial \Omega .  \label{semi_line_equ}
\end{equation*}%

Then, the approximation map $\left( {\small \varphi ,\psi }\right)
\mapsto \mathcal{T}^{\delta }\left( {\small \varphi ,\psi }\right)
$ will then be differentiable and approximate necessary conditions
will be derived, such that

\begin{equation*}
\beta _{\delta }(r)=\tfrac{1}{\delta }%
\begin{cases}
0 & \mathrm{if\ }r\geq 0 \\
-r^{2} & \mathrm{if\ }r\in \left[ -\frac{1}{2},0\right] \\
r+\frac{1}{4} & \mathrm{if\ }r\leq -\frac{1}{2}%
\end{cases}%
\end{equation*}%

where $\beta(\cdot) $ is negative and belongs to
$\mathscr{C}^{1}\left( \mathbb{R}\right) $, such that $\delta $ is
strictly positive and goes to $0$. Then $\beta _{\delta }^{\prime
}(\cdot)$ is given by

\begin{equation*}
\beta _{\delta }^{\prime }(r)=\tfrac{1}{\delta }%
\begin{cases}
0 & \mathrm{if\ }r\geq 0 \\
-2r & \mathrm{if\ }r\in \left[ -\frac{1}{2},0\right] \\
1 & \mathrm{if\ }r\leq -\frac{1}{2}%
\end{cases}%
\end{equation*}%

As $\beta _{\delta }(\cdot -\varphi )-\beta _{\delta }(\psi -\cdot
)$ is nondecreasing, it is well known (see \cite{Gilbarg}), that
boundary value
problem (\ref{semi_line_equ}) admits a unique solution $y^{\delta }$ in $%
H^{2}(\Omega )\cap H_{0}^{1}(\Omega )$ for a fixed $\varphi $ and $\psi $ in $%
H^{2}(\Omega )\cap H_{0}^{1}(\Omega )$ and $f$ in $L^{2}\left(
\Omega \right) $. In the sequel, we set $y^{\delta
}=\mathcal{T}^{\delta }\left( \varphi ,\psi \right) $ and in
addition, $c$ or $C$ denotes a general positive constant
independent of any approximation parameter. So for any $\delta
> 0 $, we define

\begin{equation}\label{cstfonn}
J_{\delta }\left( \varphi ,\psi \right)
=\tfrac{1}{2}[\int\nolimits_{\Omega
}\left( \mathcal{T}^{\delta }\left( \varphi ,\psi \right) -z\right) ^{2}dx+%
\nu  \int\nolimits_{\Omega }\left( \left( \nabla \varphi \right)
^{2}+\left( \nabla \psi \right) ^{2}\right) d{x}] \tag{$P^{\delta
}$}.
\end{equation}

Then, the approximate optimal control problem  is given by
\begin{equation}\label{cstftnD}
min\{J_{\delta }\left( \varphi ,\psi \right),\varphi ,\psi \text{
in }\mathcal{U}_{ad}\times\mathcal{U}_{ad}\} .
\end{equation}

and by using the same techniques given in \cite{Bar1} and
\cite{BergLenhBi}, the problem \eqref{cstftnD} has, at least, one
solution denoted by $(y^{\delta} ,p^{\delta}
,\varphi^{\delta},\psi^{\delta} )$ and characterized by the
following Theorem

\begin{theorem}
Since $\left( \varphi ^{\delta },\psi ^{\delta }\right) $ is an
optimal solution to $\left( \mathcal{P}^{\delta }\right) ,$ and
$y^{\delta }=\mathcal{T}^{\delta }\left( \varphi ^{\delta },\psi
^{\delta }\right).$ Then there exist $p^{\delta }\text{ in }
\mathcal{U}$, $\mu _{1}^{\delta }=\beta _{\delta }^{^{\prime
}}(y^{\delta }-\varphi ^{\delta })p^{\delta }$ and $\mu
_{2}^{\delta }=\beta_{\delta } ^{^{\prime }}(\psi ^{\delta
}-y^{\delta })p^{\delta }\text{ in }
L^{2}\left( \Omega \right) $ such that the following optimality system $%
\left(S^{\delta }\right) $ is satisfied%

\begin{equation*}
\left\{
\begin{array}{l}
Ay^{\delta }+(\beta_{\delta } (y^{\delta }-\varphi ^{\delta
})-\beta_{\delta } (\psi ^{\delta }-y^{\delta }))=f\text{ in }\Omega  \\
\\
A^{\ast }p^{\delta }+\mu _{1}^{\delta }+\mu _{2}^{\delta }=y^{\delta }-z%
\text{ in }\Omega   \\
\\

\nu \Delta \varphi ^{\delta }+\nu \Delta \psi ^{\delta
}+\beta_{\delta } {^{\prime }}(y^{\delta }-\varphi ^{\delta
})p^{\delta } + \beta _{\delta }{^{\prime }}(\psi ^{\delta
}-y^{\delta })p^{\delta }=0\\
\\
y^{\delta }=p^{\delta }=\varphi^{\delta }=\psi^{\delta }=0 \text{
on }
\partial \Omega
\end{array}%
\right.
\end{equation*}

\end{theorem}
Now, we give some important results relevant for the sequel of
this paper.

\begin{lemma}\label{lm:b1}
From the definition of $\beta(\cdot)$ and since $ p_{n}^{\delta }$
belongs to $\mathcal{B}_{H^{1}_{0}\left( \Omega \right)
}(0,\tilde{\rho_{3}})$, where $\tilde{\rho_{3}}$ is a positive
constant, and for $\left( y_{i}^{\delta },\varphi _{i}^{\delta
},p_{i}^{\delta }\right)$ in  $\tilde{\mathcal{U}}$ where
$\tilde{\mathcal{U}}$=$H^{1}_{0}\left( \Omega \right)\times
H^{2}_{0}\left( \Omega \right) \times H^{1}_{0}\left( \Omega
\right) $ and $ i=1,2$, we get

\begin{multline*}
\parallel
\beta _{\delta }^{\prime }\left( y_{2}^{\delta }-\varphi
_{2}^{\delta }\right) p_{2}^{\delta }-\beta _{\delta }^{\prime
}\left( y_{1}^{\delta }-\varphi_{1}^{\delta }\right) p_{1}^{\delta
}\parallel _{L^{2}\left( \Omega \right) } \leq
\dfrac{C}{\delta}\parallel p_{2}^{\delta }- p_{1}^{\delta
}\parallel _{H^{1}\left( \Omega \right)
}+\\
 \dfrac{C\tilde{\rho}_{3}}{\delta}\parallel y_{2}^{\delta
}-y_{1}^{\delta }\parallel _{L^{2}\left( \Omega \right) }+
\dfrac{C\tilde{\rho}_{3}}{\delta}\parallel \varphi_{2}^{\delta
}-\varphi_{1}^{\delta }\parallel _{L^{2}\left( \Omega \right) }
\end{multline*}
\end{lemma}

\begin{proof}
 By the definition of $\beta^{\prime}\left(\cdot\right)$ we get
\begin{multline*} (\beta _{\delta
}^{\prime }\left( y_{2}^{\delta }-\varphi _{2}^{\delta }\right)
p_{2}^{\delta }-\beta _{\delta }^{\prime }\left( y_{1}^{\delta
}-\varphi_{1}^{\delta }\right) p_{1}^{\delta }) = \beta _{\delta
}^{\prime }\left( y_{2}^{\delta }-\varphi _{2}^{\delta }\right)(
p_{2}^{\delta }- p_{1}^{\delta })+\\
(\beta _{\delta }^{\prime }\left( y_{2}^{\delta }-\varphi
_{2}^{\delta }\right)-\beta _{\delta }^{\prime }\left(
y_{1}^{\delta }-\varphi_{1}^{\delta }\right)) p_{1}^{\delta }.
\end{multline*}

Then, by Cauchy-Schwarz inequality and since $p_{1}^{\delta }$
belongs to $\mathcal{B}_{H^{1}\left( \Omega \right)}(0,\rho_{3}) $
and by the Mean-Value Theorem applied in the interval of sides
$\{\left( y_{2}^{\delta }-\varphi _{2}^{\delta }\right)$,$\left(
y_{1}^{\delta }-\varphi_{1}^{\delta }\right)\}$, we can deduce

\begin{multline*}
\parallel
\beta _{\delta }^{\prime }\left( y_{2}^{\delta }-\varphi
_{2}^{\delta }\right) p_{2}^{\delta }-\beta _{\delta }^{\prime
}\left( y_{1}^{\delta }-\varphi_{1}^{\delta }\right) p_{1}^{\delta
}\parallel _{L^{2}\left( \Omega \right) } \leq
\dfrac{C}{\delta}\parallel p_{2}^{\delta }- p_{1}^{\delta
}\parallel _{H^{1}\left( \Omega \right)
}+\\
 \dfrac{C\tilde{\rho}_{3}}{\delta}\parallel \left(y_{2}^{\delta
}-y_{1}^{\delta }\right)-\left( \varphi _{2}^{\delta
}-\varphi_{1}^{\delta }\right)\parallel _{L^{2}\left( \Omega
\right) }
\end{multline*}

\end{proof}
\begin{lemma}\label{lm:b2}
Let $\left( y_{i}^{\delta },\varphi _{i}^{\delta },p_{i}^{\delta
}\right)$ belong to $\tilde{\mathcal{U}}$ where $ i=1,2$ and by
the properties of $\beta_{\delta }(\cdot)$, we get

\begin{multline*}
\parallel (\beta_{\delta } (y_{2}^{\delta }-\varphi _{2}^{\delta })-\beta_{\delta }
(y_{1}^{\delta }-\varphi _{1}^{\delta }))-(\beta_{\delta } (\psi
_{2}^{\delta }-y_{2}^{\delta })-\beta_{\delta } (\psi _{1}^{\delta
}-y_{1}^{\delta }))\parallel _{L^{2}\left( \Omega
\right)} \leq\\
\dfrac{C}{\delta} \parallel y_{2}^{\delta }-y _{1}^{\delta
}\parallel _{L^{2}\left( \Omega \right)}+
\dfrac{C}{\delta}\parallel\varphi_{2}^{\delta }-\varphi
_{1}^{\delta }\parallel _{L^{2}\left( \Omega \right)} +
\dfrac{C}{\delta}\parallel \psi _{2}^{\delta }-\psi_{1}^{\delta }
\parallel _{L^{2}\left( \Omega
\right)}.
\end{multline*}
\end{lemma}

\begin{proof}
It is easy to see that
\begin{multline*}
\parallel (\beta_{\delta } (y_{2}^{\delta }-\varphi _{2}^{\delta })-\beta_{\delta }
(y_{1}^{\delta }-\varphi _{1}^{\delta }))-(\beta_{\delta } (\psi
_{2}^{\delta }-y_{2}^{\delta })-\beta_{\delta } (\psi _{1}^{\delta
}-y_{1}^{\delta }))\parallel _{L^{2}\left( \Omega
\right)} \leq\\
 \parallel \beta_{\delta } (y_{2}^{\delta }-\varphi _{2}^{\delta
})-\beta_{\delta } (y_{1}^{\delta }-\varphi _{1}^{\delta
})\parallel _{L^{2}\left( \Omega
\right)} \\
+ \parallel \beta_{\delta } (\psi _{2}^{\delta }-y_{2}^{\delta
})-\beta_{\delta } (\psi _{1}^{\delta }-y_{1}^{\delta })\parallel
_{L^{2}\left( \Omega \right)}.
\end{multline*}

By the Mean-Value Theorem applied in the interval of sides
$\{(y_{2}^{\delta }-\varphi _{2}^{\delta })$, $(y_{1}^{\delta
}-\varphi _{1}^{\delta })\}$ and $\{(\psi _{2}^{\delta
}-y_{2}^{\delta })$, $(\psi _{1}^{\delta }-y_{1}^{\delta })\}$, we
get

\begin{multline*}
\parallel (\beta_{\delta } (y_{2}^{\delta }-\varphi _{2}^{\delta })-\beta_{\delta }
(y_{1}^{\delta }-\varphi _{1}^{\delta })-(\beta_{\delta } (\psi
_{2}^{\delta }-y_{2}^{\delta })-\beta_{\delta } (\psi _{1}^{\delta
}-y_{1}^{\delta }))\parallel _{L^{2}\left( \Omega
\right)} \leq\\
\dfrac{C}{\delta} \parallel y_{2}^{\delta }-y_{1}^{\delta
}\parallel _{L^{2}\left( \Omega \right)}+
\dfrac{C}{\delta}\parallel\varphi_{2}^{\delta }-\varphi
_{1}^{\delta }\parallel _{L^{2}\left( \Omega \right)} +
\dfrac{C}{\delta}\parallel \psi _{2}^{\delta }-\psi_{1}^{\delta }
\parallel _{L^{2}\left( \Omega
\right)}.
\end{multline*}

\end{proof}

\begin{theorem}
For any triplet $\left( y_{i}^{\delta },\varphi _{i}^{\delta
},\psi _{i}^{\delta }\right) \text{ in } \tilde{\mathcal{U}}$ that
satisfies the optimality system
$\left({S}^{\delta }\right) $ where $i=1,2$, and since $\delta \leq C$, we get%
\begin{equation*}
\parallel y_{2}^{\delta }-y_{1}^{\delta }\parallel _{H^{1}\left( \Omega
\right) }\leq l_{1}(\parallel \varphi _{2}^{\delta }-\varphi
_{1}^{\delta }\parallel _{L^{2}\left( \Omega \right) }+\parallel
\psi _{2}^{\delta }-\psi _{1}^{\delta }\parallel _{L^{2}\left(
\Omega \right) }).
\end{equation*}%

where $l_{1}:=\dfrac{C}{\delta }$. This means that the mapping $y^{\delta }:=%
\mathcal{T}^{\delta }\left( \varphi ^{\delta },\psi ^{\delta
}\right) $, is Lipschitizian, with a Lipschitz constant $l_{1}$.
\end{theorem}

\begin{proof}
see \cite{BergLenhBi}
\end{proof}

\begin{lemma}
For any triplet $\left( y^{\delta },\varphi ^{\delta },\psi
^{\delta }\right) $ in $\tilde{\mathcal{U}}$, satisfying
the optimality system $(S^{\delta })$, we have%
\begin{equation*}
\parallel y^{\delta }\parallel _{H^{1}\left( \Omega\right)}\leq \max \left( C\parallel \varphi
^{\delta }\parallel _{H^{1}\left( \Omega\right)},C\parallel \psi
^{\delta }\parallel _{H^{1}\left( \Omega\right)}\right),
\end{equation*}

and moreover when $\varphi ^{\delta }$ and $\psi ^{\delta }${\small \ }belong to $%
B_{H^{2}\left( \Omega\right)}\left( 0,\rho _{1}\right) \cap
\mathcal{W},$ we deduce that

\begin{equation*}
\parallel y^{\delta }\parallel _{H^{1}\left( \Omega \right) }\leq \rho _{2}.
\label{1.4}
\end{equation*}%

This means that $y^{\delta }$ \text{ belongs to } $B_{H^{1}\left(
\Omega\right)}\left( 0,\rho _{2}\right) \cap \mathcal{U}$, where
$\rho _{2}:=C\rho _{1}$.
\end{lemma}

\begin{proof}
\cite{BergLenhBi}
\end{proof}

\begin{lemma}
For any pair $\left( p^{\delta },y^{\delta }\right) $ in
$\mathcal{U}\times \left( \mathcal{U}\cap B_{H^{1}}\left( 0,\rho
_{2}\right) \right) $,
satisfying the optimality system $(S^{\delta }),$ we have%
\begin{equation*}
\parallel p^{\delta }\parallel _{H^{1}}\leq C\parallel y^{\delta }\parallel
_{H^{1}\left( \Omega \right)},
\end{equation*}

and when $y^{\delta }${\small \ }belongs to $B_{H^{1}}\left(
0,\rho _{2}\right) \cap \mathcal{U},$ we deduce that
\begin{equation*}
\parallel p^{\delta }\parallel _{H^{1}\left( \Omega \right) }\leq \rho _{3}.
\end{equation*}%
This means that $p^{\delta }$ belongs to $B_{H^{1}}\left( 0,\rho
_{3}\right) \cap \mathcal{U}$, where $\rho _{3}:=C\rho _{2}$.
\end{lemma}

\begin{proof}
From the adjoint equation of optimality system $(S^{\delta }),$ we
have
\begin{equation*}
\sigma ^{\ast }\left( p^{\delta },v\right) +\left( \mu
_{1}^{\delta }+\mu
_{2}^{\delta },v\right) =\left( y^{\delta }-z,v\right), \text{ for all }v%
\text{ in }H_{0}^{1}\left( \Omega \right)
\end{equation*}

if we take $v=p^{\delta}$, and by the coercivity condition
H\textsubscript{2} of $\sigma
^{\ast }\left(\cdot,\cdot\right)$, we obtain%
\begin{equation*}
\parallel p^{\delta }\parallel _{H^{1}\left( \Omega \right)}\leq C\parallel y^{\delta }\parallel
_{H^{1}\left( \Omega \right)}.
\end{equation*}
\end{proof}

\section{ Convergence study of an iterative algorithm}

In this section, we give an algorithm to solve problem $(P^{\delta
})$. Roughly speaking, we propose an implicit algorithm to solve
the necessary optimality system $(S^{\delta })$. The proposed
algorithm is based on the Gauss-Seidel method and is given below.

\begin{algorithm}[h]
\label{Alg1}%
\caption{Gauss-Seidel algorithm (Continuous version)}
\begin{algorithmic}[1]

\STATE
 \textbf{Input :} $\left\{ y_{0}^{\delta },p_{0}^{\delta },\varphi
_{0}^{\delta },\psi _{0}^{\delta },\lambda _{0}^{\delta },\delta
,\nu ,\varepsilon \right\} $ choose $\varphi
_{0}^{\delta },\psi _{0}^{\delta }$ in $\mathcal{U},\varepsilon $ and $\delta $ in $%
R_{+}^{\ast }$;

\STATE\textbf{Begin:}\\

\STATE \textbf{Solve} $Ay_{n}^{\delta }+\frac{1}{\delta }\left(
\beta \left( y_{n}^{\delta }-\varphi _{n-1}^{\delta }\right)
-\beta \left( \psi _{n-1}^{\delta }-y_{n}^{\delta }\right) \right)
=f\text{ on }y_{n}^{\delta }$

\STATE \textbf{Solve } $\left( A+\left( \beta _{\delta }^{\prime
}\left( y_{n}^{\delta }-\varphi _{n-1}^{\delta }\right) +\beta
_{\delta }^{\prime }\left( \psi _{n-1}^{\delta }-y_{n}^{\delta
}\right) \right) \right) p_{n}^{\delta }=y_{n}^{\delta }-z$ on
$p_{n}^{\delta }$.

 \STATE \textbf{Calculate }   $\lambda _{n}^{\delta }=\nu \Delta \varphi _{n-1}^{\delta
}+\beta _{\delta }^{\prime }\left( y_{n}^{\delta }-\varphi
_{n-1}^{\delta }\right) p_{n}^{\delta }$

\STATE \textbf{Solve } $\nu \Delta \psi _{n}^{\delta }+\beta
_{\delta }^{\prime }\left( \psi _{n}^{\delta }-y_{n}^{\delta
}\right) p_{n}^{\delta }=-\lambda _{n}^{\delta }$ on $\psi
_{n}^{\delta }$.

\STATE \textbf{Solve } $-\lambda _{n}^{\delta }+\nu \Delta \varphi
_{n}^{\delta }+\beta _{\delta }^{\prime }\left( y_{n}^{\delta
}-\varphi _{n}^{\delta }\right)
p_{n}^{\delta }=0$ on $\varphi _{n}^{\delta }.$ \\
\STATE \textbf{If } the stop criteria is fulfilled \textbf{Stop.}

\STATE \textbf{Ensure :} $s_{n}^{\delta }:=\left( y_{n}^{\delta
},\varphi _{n}^{\delta },\psi _{n}^{\delta },p_{n}^{\delta
}\right) $ \textbf{is a solution}

\STATE \ \ \ \ \ \ \ \ \  \textbf{Else; } $n\leftarrow n+1$,
\textbf{Go to} \textbf{Begin.}

\STATE \textbf{End if}

\STATE \textbf{End algorithm}.
\end{algorithmic}
\end{algorithm}
This algorithm can be seen as a successive approximation method to
compute the five points of the function $F$ that we are going to
define. From the
different steps of the above algorithm, we define the following functions $%
F_{i}$, for $i=1,\,2,\,3,\,4$ as

\begin{itemize}
\item From step 1, we define $F_{1}:\mathcal{U}\times \mathcal{U}\rightarrow \mathcal{U}$, such that%

\begin{equation*}
y_{n}^{\delta }:=F_{1}\left( \varphi _{n-1}^{\delta },\psi
_{n-1}^{\delta }\right),  \label{2.1}
\end{equation*}%

we see that $F_{1}$ depends on $\varphi _{n-1}^{\delta },\psi
_{n-1}^{\delta }$, and gives $y_{n}^{\delta }$ as the solution of
the following state equation

\begin{equation}
Ay_{n}^{\delta }+\beta _{\delta }\left( y_{n}^{\delta }-\varphi
_{n-1}^{\delta }\right) -\beta _{\delta }\left( \psi
_{n-1}^{\delta }-y_{n}^{\delta }\right) =f \text{ in } \Omega,\text{ and
}y_{n}^{\delta }=0\text{ on }\partial \Omega . \label{2.2}
\end{equation}%

\item From step 2, we define $F_{2}:\mathcal{U}\times
\mathcal{U}\times \mathcal{U}\rightarrow \mathcal{U}$, such that
\begin{equation}
p_{n}^{\delta }:=F_{2}\left( y_{n}^{\delta },\varphi
_{n-1}^{\delta },\psi _{n-1}^{\delta }\right),  \label{2.3}
\end{equation}%

we see that $F_{2}$ depends on $\varphi _{n-1}^{\delta },\psi
_{n-1}^{\delta }$ and $y_{n}^{\delta }$, and gives $p_{n}^{\delta
}$ as the solution of the following adjoint state equation

\begin{equation}
Ap_{n}^{\delta }+\beta _{\delta }^{\prime }\left( y_{n}^{\delta
}-\varphi _{n-1}^{\delta }\right) p_{n}^{\delta }+\beta _{\delta
}^{\prime }\left( \psi _{n-1}^{\delta }-y_{n}^{\delta }\right)
p_{n}^{\delta }=y_{n}^{\delta }-z \text{ in } \Omega,\text{ and }\text{
}p_{n}^{\delta }=0\text{ on }\text{ }\partial \Omega . \label{2.4}
\end{equation}%

\item From step 3, we define $F_{3}: \mathcal{U}\times \mathcal{U}\rightarrow \mathcal{U}$, such that%

\begin{equation*}
\psi _{n}^{\delta }:=F_{3}\left( y_{n}^{\delta },p_{n}^{\delta
}\right), \label{2..7}
\end{equation*}%

we see that $F_{3}$ depends on $p_{n}^{\delta }$ and
$y_{n}^{\delta }$, and since $\psi _{n}^{\delta }$ is given, we define $\lambda _{n}^{\delta }$   by the following equation

\begin{equation}
-\lambda _{n}^{\delta }=\nu \Delta \psi _{n}^{\delta }+\beta
_{\delta }^{\prime }\left( \psi
_{n}^{\delta }-y_{n}^{\delta }\right) p_{n}^{\delta } \text{ in } \Omega,\text{ and }
\psi _{n}^{\delta }=0\text{ on }\partial \Omega .
\label{2..8}
\end{equation}%

\item From step 4, we define $F_{4}:\mathcal{U}\times \mathcal{ U}\rightarrow\mathcal{ U}$, such that%

\begin{equation*}
\varphi _{n}^{\delta }:=F_{4}\left( y_{n}^{\delta },p_{n}^{\delta
}\right) , \label{2..9}
\end{equation*}%

we see that $F_{4}$ depends on $p_{n}^{\delta }$, and
$y_{n}^{\delta }$, and since $\varphi _{n}^{\delta }$ is given, $\lambda _{n}^{\delta }$ can be also defined by the following equation

\begin{equation}
-\lambda _{n}^{\delta }+\nu \Delta \varphi _{n}^{\delta }+\beta
_{\delta }^{\prime }\left(y_{n}^{\delta }- \varphi
_{n}^{\delta }\right) p_{n}^{\delta }=0\text{ in } \Omega,\text{ and }%
\varphi _{n}^{\delta }=0\text{ on }\partial \Omega .
\label{2..10}
\end{equation}%
\end{itemize}

\begin{remark}
We note that the equation given by \eqref{2..8} is only used to
solve the equation given by \eqref{2..10}.
\end{remark}

Then according the above definitions of $F_{i},$ where
$i=1,\,2,\,3,\,4$, let us define the map $F:\mathcal{U}\times
\mathcal{U}\rightarrow \mathcal{U}\times \mathcal{U}$, as

\begin{equation*}
\left( \varphi _{n}^{\delta },\psi _{n}^{\delta }\right) :=F\left(
\varphi _{n-1}^{\delta },\psi _{n-1}^{\delta }\right) ,
\end{equation*}%

where

\begin{equation*}
 \varphi _{n}^{\delta } :=\tilde{F}_{1}\left( \varphi _{n-1}^{\delta },\psi
_{n-1}^{\delta }\right) :=F_{4}\left( F_{1}\left( \varphi
_{n-1}^{\delta },\psi _{n-1}^{\delta }\right) ,F_{2}\left(
F_{1}\left( \varphi _{n-1}^{\delta },\psi _{n-1}^{\delta }\right)
,\varphi _{n-1}^{\delta },\psi _{n-1}^{\delta }\right) \right),
\end{equation*}

and

\begin{equation*}
\psi _{n}^{\delta } :=\tilde{F}_{2}\left( \varphi _{n-1}^{\delta
},\psi _{n-1}^{\delta }\right):=F_{3}\left( F_{1}\left( \varphi
_{n-1}^{\delta },\psi _{n-1}^{\delta }\right) ,F_{2}\left(
F_{1}\left( \varphi _{n-1}^{\delta },\psi _{n-1}^{\delta }\right)
,\varphi _{n-1}^{\delta },\psi _{n-1}^{\delta
}\right) \right)%
\end{equation*}%

such that

\begin{equation*} F\left( \varphi _{n-1}^{\delta },\psi
_{n-1}^{\delta }\right) =\left(\tilde{F}_{1}\left( \varphi
_{n-1}^{\delta },\psi _{n-1}^{\delta }\right),\tilde{F}_{2}\left(
\varphi _{n-1}^{\delta },\psi _{n-1}^{\delta }\right)\right)
\end{equation*}

\begin{proposition}
 Let $\varphi _{n-1}^{\delta }$ and $\psi
_{n-1}^{\delta }$belong to $\mathcal{U}$ and $\left( y_{n}^{\delta
},\varphi _{n}^{\delta },\psi _{n}^{\delta },p_{n}^{\delta
}\right)$ satisfies equations \eqref{2.2}, \eqref{2.4},
\eqref{2..8} and \eqref{2..10} given respectively by $F_{1}$,
$F_{2}$, $F_{3}$ and $F_{4}$  such that $\delta \leq C$, then we
get

\begin{equation}
\ \parallel y_{n}^{\delta }\parallel _{H^{1}\left( \Omega \right)
}\leq \frac{c}{\delta }\parallel \varphi _{n-1}^{\delta }\parallel
_{H^{2}\left( \Omega \right) }+\frac{c}{\delta }\parallel \psi
_{n-1}^{\delta }\parallel _{H^{2}\left( \Omega \right)
}+c\parallel f\parallel _{L^{2}\left( \Omega \right) } \label{2.9}
\end{equation}%

\begin{equation}
\parallel p_{n}^{\delta }\parallel _{H^{1}\left( \Omega \right) }\leq
C+C\parallel y_{n}^{\delta }\parallel _{{H^{1}\left( \Omega
\right) }} \label{2.10}
\end{equation}

\begin{equation}
\parallel \varphi _{n}^{\delta }\parallel _{H^{2}\left( \Omega \right) }\leq
\tfrac{C}{\delta \nu }\parallel p_{n}^{\delta }\parallel
_{H^{1}\left( \Omega \right) }+C  \label{2.11}
\end{equation}%

and

\begin{equation}
\parallel \psi _{n}^{\delta }\parallel _{H^{2}\left( \Omega \right) }\leq
\tfrac{C}{\delta \nu }\parallel p_{n}^{\delta }\parallel
_{H^{1}\left( \Omega \right) }+C  \label{2..12}
\end{equation}%

\end{proposition}

\begin{proof}

From the state equation \eqref{2.2}, we obtain%

\begin{multline}
c\parallel y_{n}^{\delta }\parallel _{H^{1}\left( \Omega \right)
}^{2}\leq \frac{c}{\delta }\left( c\parallel y_{n}^{\delta
}\parallel _{H^{1}\left( \Omega \right) }+\parallel \varphi
_{n-1}^{\delta }\parallel _{L^{2}\left( \Omega \right) }\right)
\parallel \varphi _{n-1}^{\delta }\parallel _{L^{2}\left( \Omega
\right) }+  \\ +\frac{c}{\delta }\left( \parallel \psi
_{n-1}^{\delta }\parallel _{L^{2}\left( \Omega \right)
}+c\parallel y_{n}^{\delta }\parallel _{H^{1}\left( \Omega \right)
}\right) \parallel \psi _{n-1}^{\delta }\parallel _{L^{2}\left(
\Omega \right) }+c\parallel f\parallel _{L^{2}\left( \Omega
\right) }\parallel y_{n}^{\delta }\parallel _{H^{1}\left( \Omega
\right) } \label{IV2.12}
\end{multline}

Then from the above inequality \eqref{IV2.12}, we deduce
\begin{equation*}
\parallel y_{n}^{\delta }\parallel _{H^{1}\left( \Omega \right) }\leq
\left( \frac{c}{\delta }\parallel \varphi _{n-1}^{\delta
}\parallel _{L^{2}\left( \Omega \right) }+\frac{c}{\delta
}\parallel \psi _{n-1}^{\delta }\parallel _{L^{2}\left( \Omega
\right) }+c\parallel f\parallel _{L^{2}\left( \Omega \right)
}\right).
\end{equation*}

Now, from the adjoint state equation \eqref{2.4}, by the coercivity condition of $\sigma ^{\ast }\left(
\cdot,\cdot\right) $ given by H\textsubscript{2}, we obtain that%

\begin{equation*}
\parallel p_{n}^{\delta }\parallel _{H^{1}\left( \Omega \right) }\leq
C\parallel y_{n}^{\delta }\parallel _{{H^{1}\left( \Omega \right)
}},
\end{equation*}%

by using the following equation

\begin{equation*}
\lambda_{n}^{\delta }=\nu\Delta\varphi _{n-1}^{\delta } + \beta
_{\delta }^{\prime }\left( y_{n}^{\delta }-\varphi _{n-1}^{\delta
}\right) p_{n}^{\delta }\text{ and }\text{ }\varphi_{n}^{\delta
}=0\text{ on } \partial \Omega
\end{equation*}%

we deduce that

\begin{equation*}
\parallel \lambda _{n}^{\delta }\parallel _{H^{1}\left( \Omega \right) }\leq
\tfrac{C}{\delta  }\parallel p_{n}^{\delta }\parallel
_{H^{1}\left( \Omega \right) }+C \nu \parallel
\varphi_{n-1}^{\delta }\parallel _{H^{2}\left( \Omega \right)}.
\end{equation*}%

 From equation \eqref{2..8}, and by the coercivity condition H\textsubscript{2} of
$\sigma\left(\cdot,\cdot\right) $, and the definition of $\beta
_{\delta }^{\prime }\left( \cdot\right)$, we obtain

\begin{equation*}
\parallel \psi _{n}^{\delta }\parallel _{H^{2}\left( \Omega \right) }\leq \tfrac{C}{\delta \nu }\parallel p_{n}^{\delta }\parallel
_{H^{1}\left( \Omega \right) }+C\parallel \varphi_{n-1}^{\delta
}\parallel _{H^{2}\left( \Omega \right) }.
\end{equation*}%

Using equation \eqref{2..10}, and by the coercivity condition
H\textsubscript{2} of $\sigma\left( .,.\right) $, and the
definition of $\beta _{\delta }^{\prime }\left( \cdot\right)$, we
get

\begin{equation}
\parallel \varphi _{n}^{\delta }\parallel _{H^{2}\left( \Omega \right) }\leq \tfrac{C}{\delta \nu }\parallel p_{n}^{\delta }\parallel
_{H^{1}\left( \Omega \right) }+C\parallel \varphi_{n-1}^{\delta
}\parallel _{H^{2}\left( \Omega \right) }.
\label{c1}
\end{equation}%
\end{proof}

\begin{corollary}\label{cor:c1}
Since $\varphi _{n-1}^{\delta }$ and $\psi _{n-1}^{\delta }$ belong to $%
B_{H^{2}}\left( 0,\tilde{\rho}_{1}\right) \cap \mathcal{U}$, and
letting $\left( y_{n}^{\delta },\varphi _{n-1}^{\delta },\psi
_{n-1}^{\delta }\right)$ belong to $\tilde{ \mathcal{U}}$ to
satisfy the conditions \eqref{2.2}, \eqref{2..8} and \eqref{2..10}
given respectively by $F_{1}$, $F_{3}$ and $F_{4}$  such that
$\delta \leq C$, then we get

\begin{equation*}
\parallel y_{n}^{\delta }\parallel _{H^{1}\left( \Omega \right) }\leq \tilde{%
\rho}_{2}
\end{equation*}%

This means that $y_{n}^{\delta }$ belongs to $B_{H^{1}}\left( 0,\tilde{\rho}%
_{2}\right) \cap \mathcal{U},$ where $\tilde{\rho}_{2}:=\left( C+\dfrac{C}{\delta }%
\tilde{\rho}_{1}\right)$.
\end{corollary}

\begin{proof}
It's obvious by using \eqref{2.9}.
\end{proof}

\begin{corollary}\label{cor:c3}
Since the hypotheses of Corollary \ref{cor:c1}  are fulfilled, and
by letting $ (y_{n}^{\delta },p _{n}^{\delta })$ $\in$
$H^1(\Omega)\times H^1(\Omega)$ to satisfy the
conditions \eqref{2.2}, \eqref{2.4} given respectively by $%
F_{1}$, $F_{2}$, we get

\begin{equation*}
\parallel p_{n}^{\delta }\parallel _{H^{1}\left( \Omega \right) }\leq \tilde{%
\rho}_{3}  \label{2.13}
\end{equation*}

This means that $p_{n}^{\delta }$ belongs to $B_{H^{1}}\left( 0,\tilde{\rho}%
_{3}\right) \cap \mathcal{U},$ where
$\tilde{\rho}_{3}:=C\tilde{\rho}_{2}$.
\end{corollary}

\begin{proof}
It's obvious from inequalities \eqref{2.10} and \eqref{c1}.
\end{proof}

\begin{corollary}
Since the hypotheses of corollary $\ref{cor:c1} $ are fulfilled,
and by letting $\left( y_{n}^{\delta },\varphi _{n}^{\delta },\psi
_{n}^{\delta }\right) $ in $\tilde{\mathcal{U}}$ to satisfy the
conditions \eqref{2.2}, \eqref{2..8} and \eqref{2..10} given respectively by $%
F_{1}$, $F_{3}$ and $F_{4}$, we get

\begin{equation*}
\parallel \varphi _{n}^{\delta }\parallel _{H^{2}\left( \Omega \right) }\leq
\tilde{\rho}_{4}.  \label{2.14}
\end{equation*}%

and%

\begin{equation*}
\parallel \psi _{n}^{\delta }\parallel _{H^{2}\left( \Omega \right) }\leq
\tilde{\rho}_{4}.
\end{equation*}

This means that $\varphi _{n}^{\delta }$ and $\psi _{n}^{\delta }$
belong respectively to $\mathcal{B}_{H^{2}}\left(
0,\tilde{\rho}_{4}\right) \cap \mathcal{U}$, and
$\tilde{\rho}_{4}:=\dfrac{C}{\delta \nu
}\tilde{\rho}_{3}+C\tilde{\rho}_{1}$, where $\tilde{\rho}_{3}$,
$\tilde{\rho}_{2}$ are given respectively by corollaries
$\ref{cor:c3}$ and \ref{cor:c1}.
\end{corollary}

\begin{proof}
It's obvious from inequalities \eqref{2.11}, \eqref{2..12} and \eqref{2.13}.
\end{proof}

Let us give the following theorem to show that the mapping $F$ is
locally Lipschitz.

\begin{theorem}\label{thr1}
If $\delta \leq C$, then the mapping $F$ is locally Lipschitz from
$\mathcal{B}_{H^{2}}\left( 0,\tilde{\rho}_{1}\right) \cap
\mathcal{U}$ $\times
\mathcal{B}_{H^{2}}\left( 0,\tilde{\rho}_{1}\right) \cap \mathcal{U}$ to $\mathcal{B}_{H^{2}}\left( 0,%
\tilde{\rho}_{4}\right) \cap \mathcal{U}\times \mathcal{B}_{H^{2}}\left( 0,\tilde{\rho}%
_{4}\right) \cap \mathcal{U},$ with the Lipschitz constant $%
l:=l_{1}(l_{3}+l_{4})+l_{2}(l_{3}+l_{4})+l_{1}l_{2}(l_{3}+l_{4})$, where $\tilde{\rho}_{4}=\dfrac{C}{%
\delta \nu }+(\dfrac{C}{\delta ^{2}\nu }+C)\tilde{\rho}_{1},$ $l_{1}:=\dfrac{C}{%
\delta },$ $l_{2}:=\left( C+\dfrac{C\tilde{\rho}_{3}}{\delta }\right) ,$ $%
l_{3}=l_{4}:=\dfrac{C}{\delta \nu C-C\tilde{\rho}_{3}}$, and
$\tilde{\rho}_{3}$ is given by Corollary \ref{cor:c3}.
\end{theorem}

To prove the previous theorem, we need the followings Lemmas.

\begin{lemma}\label{lm:l1}
The function $F_{1}$ defined by \eqref{2.1} is Lipschitz
continuous from $\mathcal{U}$ to $\mathcal{U}$, with a Lipschitz constant $l_{1}:=\dfrac{C}{%
\delta }.$

\end{lemma}

\begin{proof}
Let $y_{n}^{\delta }=F_{1}\left( \varphi _{n-1}^{\delta ,1},\psi
_{n-1}^{\delta ,1}\right) $ and $z_{n}^{\delta }=F_{1}\left(
\varphi _{n-1}^{\delta ,2},\psi _{n-1}^{\delta ,2}\right) $, where
$\left( y_{n}^{\delta },\varphi _{n-1}^{\delta ,1},\psi
_{n-1}^{\delta ,1}\right) $ and $\left( z_{n}^{\delta },\varphi
_{n-1}^{\delta ,2},\psi _{n-1}^{\delta
,2}\right) $ belong to $\mathcal{U}\times \mathcal{U}\times \mathcal{U}$. From the equation given by \eqref{2.2}, by the coercivity condition H\textsubscript{2} of $\sigma \left(
\cdot,\cdot\right)$ and Lemma \ref{lm:b2}, we get

\begin{equation*}
\parallel y_{n}^{\delta }-z_{n}^{\delta }\parallel _{H^{1}\left( \Omega
\right) }\leq \frac{C}{\delta }\left( \parallel \varphi
_{n-1}^{\delta ,1}-\varphi _{n-1}^{\delta ,2}\parallel
_{L^{2}\left( \Omega \right) }+
\parallel \psi _{n-1}^{\delta ,1}-\psi _{n-1}^{\delta ,2}\parallel
_{L^{2}\left( \Omega \right) }\right).
\end{equation*}

\end{proof}

\begin{lemma}\label{lm:l2}
The function $F_{2}$ defined by \eqref{2.3}, is locally Lipschitz
from \\ $\left( \mathcal{B}_{H^{1}}\left(
0,\tilde{\rho}_{2}\right) \cap
\mathcal{U}\right) \times \left( \mathcal{B}_{H^{2}}\left( 0,\tilde{\rho}%
_{1}\right) \cap \mathcal{U}\right) \times \left(
\mathcal{B}_{H^{2}}\left(
0,\tilde{\rho}_{1}\right) \cap \mathcal{U}\right) $ to $\mathcal{B}_{H^{1}}\left( 0,%
\tilde{\rho}_{3}\right) \cap \mathcal{U}$, with the Lipschitz
constant $l_{2}:=\left( C+\dfrac{C\tilde{\rho}_{3}}{\delta
}\right) $, where $\tilde{\rho}_{3}$ is given by Corollary
\ref{cor:c3}.
\end{lemma}

\begin{proof}
Let $p_{n}^{\delta ,1}=F_{2}\left( y_{n}^{\delta ,1},\varphi
_{n-1}^{\delta ,1},\psi _{n-1}^{\delta ,1}\right)$ and
$p_{n}^{\delta ,2}=F_{2}\left( y_{n}^{\delta ,2},\varphi
_{n-1}^{\delta ,2},\psi _{n-1}^{\delta ,2}\right)$ where $\left(
y_{n}^{\delta ,1},\varphi _{n-1}^{\delta ,1},\psi _{n-1}^{\delta
,1}\right) $ and $\left( y_{n}^{\delta ,2},\varphi _{n-1}^{\delta
,2},\psi _{n-1}^{\delta ,2}\right) $ belong to $\left(
\mathcal{B}_{H^{1}}\left( 0,\tilde{\rho}_{2}\right) \cap
\mathcal{U}\right) \times \left( \mathcal{B}_{H^{2}}\left(
0,\tilde{\rho}_{1}\right) \cap
\mathcal{W}\right) \times \left( \mathcal{B}_{H^{2}}\left( 0,\tilde{\rho}%
_{1}\right) \cap \mathcal{W}\right)$. Then by the adjoint state
equation \eqref{2.4},  we get

\begin{multline*}
\parallel p_{n}^{\delta ,2}-p_{n}^{\delta ,1}\parallel
_{H^{1}\left(
\Omega \right) }\leq \left( C+\frac{C\text{ }\tilde{\rho}_{3}}{\delta }%
\right) (\Vert y_{n}^{\delta ,2}-y_{n}^{\delta ,1}\Vert
_{L^{2}\left( \Omega
\right) }+ \\
\Vert \psi _{n-1}^{\delta ,2}-\psi _{n-1}^{\delta ,1}\Vert
_{L^{2}\left( \Omega \right) }+\Vert \varphi _{n-1}^{\delta
,2}-\varphi _{n-1}^{\delta ,1}\Vert _{L^{2}\left( \Omega \right)
}).
\end{multline*}

\end{proof}

\begin{lemma}\label{lm:l3}
Since the following condition,%
\begin{equation*}
\tilde{\rho}_{3}\leq \delta \nu C,
\end{equation*}

is fulfilled, the function $F_{3}$
is locally Lipschitiz  from $\left( \mathcal{B}_{H^{1}}\left( 0,%
\tilde{\rho}_{2}\right) \cap \mathcal{U}\right) \times \left( \mathcal{B}%
_{H^{1}}\left( 0,\tilde{\rho}_{3}\right) \cap \mathcal{U}\right)\times \left( \mathcal{B}_{H^{1}}\left( 0,%
\tilde{\rho}_{4}\right) \cap \mathcal{U}\right) $ to $%
\left( \mathcal{B}_{H^{2}}\left( 0,\tilde{\rho}_{4}\right) \cap \mathcal{W}%
\right)$, with Lipschitz constant
\begin{equation*}
l_{3}:=\dfrac{C}{\delta \nu C-C\tilde{\rho}_{3}}.
\end{equation*}
\end{lemma}

\begin{proof}
From equation \eqref{2..8}, by the coercivity condition H\textsubscript{2} of $\sigma \left(
.,.\right)$, and by Lemma $\ref{lm:b1}$, we obtain
\begin{multline}
\parallel \varphi _{n}^{\delta ,1}-\varphi _{n}^{\delta ,2}\parallel
_{H^{2}(\Omega )} \leq \dfrac{C\text{ }\tilde{\rho}_{3}}{\nu
\delta }
\parallel (y_{n}^{\delta ,1}-y_{n}^{\delta ,2})-(\varphi _{n}^{\delta
,1}-\varphi _{n}^{\delta ,2})\parallel _{L_{2}(\Omega )}
+\dfrac{C}{\nu \delta } \parallel p_{n}^{\delta ,1}-p_{n}^{\delta
,2}\parallel _{L_{2}(\Omega )}+\\
\dfrac{C}{\nu }\parallel \lambda _{n}^{\delta ,1}-\lambda
_{n}^{\delta ,2}\parallel _{L_{2}(\Omega )}
\end{multline}

For the previous inequality to have a meaning, we must have
\begin{equation*}
\tilde{\rho}_{3}\leq C\nu \delta.
\end{equation*}
 Then, we get
\begin{multline*}
\parallel
\varphi _{n}^{\delta ,1}-\varphi _{n}^{\delta ,2}\parallel
_{H^{2}(\Omega )} \leq \dfrac{C\tilde{\rho}_{3}}{\nu \delta-
C\tilde{\rho}_{3}}
\parallel y_{n}^{\delta ,1}-y_{n}^{\delta ,2}\parallel _{L_{2}(\Omega )}+
\dfrac{C}{\nu \delta- C\tilde{\rho}_{3}} \parallel p_{n}^{\delta
,1}-p_{n}^{\delta ,2}\parallel _{L_{2}(\Omega )}+\\
+ \dfrac{ \delta C}{\nu \delta- C\tilde{\rho}_{3}}\parallel
\lambda _{n}^{\delta ,1}-\lambda _{n}^{\delta ,2}\parallel
_{L_{2}(\Omega )}.
\end{multline*}
\end{proof}

\begin{lemma}\label{lm:l4}
Since the following condition,%
\begin{equation*}
\tilde{\rho}_{3}\leq \delta \nu C,
\end{equation*}

is fulfilled, then, the function $F_{4}$ given by \eqref{2..7}
is locally Lipschitiz from \\
$\left( \mathcal{B}_{H^{1}}\left( 0,%
\tilde{\rho}_{2}\right) \cap \mathcal{U}\right) \times \left( \mathcal{B}%
_{H^{1}}\left( 0,\tilde{\rho}_{3}\right) \cap \mathcal{U}\right)\times \left( \mathcal{B}%
_{H^{1}}\left( 0,\tilde{\rho}_{4}\right) \cap \mathcal{U}\right) $ to $%
\left( \mathcal{B}_{H^{2}}\left( 0,\tilde{\rho}_{4}\right) \cap \mathcal{W}%
\right) $, with Lipschitz constant
\begin{equation*}
l_{4}:=\dfrac{C}{\delta \nu C-C\tilde{\rho}_{3}}.
\end{equation*}
\end{lemma}

\begin{proof}
From equation \eqref{2..8}, by the coercivity condition H\textsubscript{2} of $\sigma \left(
\cdot,\cdot\right)$
 and by Lemma \ref{lm:b2}, we get

\begin{multline*}
\parallel \psi _{n}^{\delta ,1}-\psi _{n}^{\delta ,2}\parallel
_{H^{2}(\Omega )} \leq \dfrac{C\text{ }\tilde{\rho}_{3}}{\nu
\delta }
\parallel (y_{n}^{\delta ,1}-y_{n}^{\delta ,2})-(\psi _{n}^{\delta
,1}-\psi _{n}^{\delta ,2})\parallel _{L_{2}(\Omega )}
+\dfrac{C}{\nu \delta } \parallel p_{n}^{\delta ,1}-p_{n}^{\delta
,2}\parallel _{L_{2}(\Omega )}+\\
+\dfrac{C}{\nu }\parallel \lambda_{n}^{\delta
,1}-\lambda_{n}^{\delta ,2}\parallel _{L_{2}(\Omega )}.
\end{multline*}

For the previous inequality to have a sense, we must have
\begin{equation*}
\tilde{\rho}_{3}\leq C\nu \delta.
\end{equation*}
 Then, we get
\begin{multline*}
\parallel
\psi _{n}^{\delta ,1}-\psi _{n}^{\delta ,2}\parallel
_{H^{2}(\Omega )} \leq \dfrac{C\tilde{\rho}_{3}}{\nu \delta-
C\tilde{\rho}_{3}}
\parallel y_{n}^{\delta ,1}-y_{n}^{\delta ,2}\parallel _{L_{2}(\Omega )}+
\dfrac{C}{\nu \delta- C\tilde{\rho}_{3}} \parallel p_{n}^{\delta
,1}-p_{n}^{\delta ,2}\parallel _{L_{2}(\Omega )}+\\
+ \dfrac{ \delta C}{\nu \delta- C\tilde{\rho}_{3}} \parallel
\lambda_{n}^{\delta ,1}-\lambda_{n}^{\delta ,2}\parallel
_{L_{2}(\Omega )}.
\end{multline*}
\end{proof}

Now, we give the proof of Theorem \ref{thr1}.

\begin{proof}

Let
\begin{equation*}(\varphi _{n}^{\delta,1 },\psi _{n}^{\delta,1
}):=(F_{3}\left( p_{n}^{\delta,1 },y_{n}^{\delta,1
}\right),F_{4}\left( p_{n}^{\delta,1 },y_{n}^{\delta,1 }\right)),
\end{equation*}
and
\begin{equation*}
(\varphi _{n}^{\delta,2 },\psi _{n}^{\delta,2 }):=(F_{3}\left(
p_{n}^{\delta,2 },y_{n}^{\delta,2 }\right),F_{4}\left(
p_{n}^{\delta,2 },y_{n}^{\delta,2 }\right)).%
\end{equation*}

 Thanks to the Lemmas \ref{lm:l3} and \ref{lm:l4}, we get

\begin{equation*}
\parallel (\varphi _{n}^{\delta,1 },\psi _{n}^{\delta,1
})-(\varphi _{n}^{\delta,2 },\psi _{n}^{\delta,2 })\parallel
_{H^2(\Omega)} \leq (l_{3}+l_{4})\left( \parallel y_{n}^{\delta,1
}-y_{n}^{\delta,2 }\parallel _{H^1(\Omega)}+\parallel
p_{n}^{\delta,1 }-p_{n}^{\delta,2 }\parallel
_{H^1(\Omega)}\right),
\end{equation*}

where $p_{n}^{\delta,1 }:=F_{2}\left( \varphi _{n-1}^{\delta,1
},y_{n}^{\delta,1 },\psi _{n-1}^{\delta,1 }\right),
p_{n}^{\delta,2 }:=F_{2}\left( \varphi _{n-1}^{\delta,1
},y_{n}^{\delta,2 },\psi _{n-1}^{\delta,1 }\right),
y_{n}^{\delta,1 }:=F_{1}(\varphi _{n-1}^{\delta,1 },\psi
_{n-1}^{\delta,1 })$, and $y_{n}^{\delta,2 }:=F_{1}(\varphi
_{n-1}^{\delta,2 },\psi _{n-1}^{\delta,2 })$, and
 by Lemmas  \ref{lm:l1} and \ref{lm:l2}, we obtain

\begin{equation*}
\parallel (\varphi _{n}^{\delta,1 },\psi _{n}^{\delta,1
})-(\varphi _{n}^{\delta,2 },\psi _{n}^{\delta,2 })\parallel
_{H^2(\Omega)} \leq l\left( \parallel \varphi_{n-1}^{\delta,1
}-\varphi_{n-1}^{\delta,2 }\parallel _{H^1(\Omega)}+\parallel
\psi_{n-1}^{\delta,1 }-\psi_{n-1}^{\delta,2 }\parallel
_{H^1(\Omega)}\right),
\end{equation*}

where
$l:=l_{1}l_{2}(l_{4}+l_{3})+l_{2}(l_{4}+l_{3})+l_{1}(l_{4}+l_{3})$
is the Lipschitz constant of the function $F$.
\end{proof}

\begin{remark} From above, we have proven that the function $F$ is
locally Lipschitz, and we can see that it is very difficult to get
a sharp estimate of the Lipschitz constant $l$ of $F$. But we are
convinced that appropriate choices of $\tilde{\rho}_{1}$ and
$\delta$ (small enough) could make this constant strictly less
than $1$, so that $F$ is contractive.
\end{remark}

In the sequel, we illustrate how the combined direct and dumped
Newton method can be used most effectively for solving the
optimality system $({S_\delta})$. The main idea is to linearize
equations given by \eqref{2.2}, \eqref{2..7} and \eqref{2..10},
for the numerical solution of the set equation \eqref{2.2},
\eqref{2..7} and \eqref{2..10}. We use the iterative relaxed
Newton's method (see \cite{Ghanemzireg}) on each mapping
$F_1$,$F_3$ and $F_4$, and prove the convergence of the proposed
algorithm.

\begin{theorem}
Since $(\bar{\varphi}^{\delta },\bar{\psi}^{\delta })$ belongs to
$ \mathcal{U}\times\mathcal{U}$ is solution of the following
equation

\begin{equation*}
\ (\bar{\varphi}^{\delta },\bar{\psi}^{\delta })-F\left(
\bar{\varphi}^{\delta },\bar{\psi}^{\delta }\right) =0 \label{3.1}
\end{equation*}

Then $\left( \bar{y}^{\delta },\bar{p} ^{\delta
},\bar{\varphi}^{\delta },\bar{\psi}^{\delta }\right)$ belonging
to $\mathcal{U} \times \mathcal{U} \times \mathcal{W}\times
\mathcal{W}$ satisfies the optimality system $({S}^\delta)$,
where, in the sequel, we put $\bar{s}^{\delta }$ $:=$ $\left(
\bar{y}^{\delta },\bar{p} ^{\delta },\bar{\varphi}^{\delta
},\bar{\psi}^{\delta }\right)$.
\end{theorem}

\begin{proof}
Since $\left(\bar{\varphi}^{\delta },\bar{\psi}^{\delta }\right)$
belonging to $\mathcal{W}\times \mathcal{W}$ satisfies equation
\eqref{3.1}, where $\left(\bar{\varphi}^{\delta
},\bar{\psi}^{\delta }\right)$ is given by

\begin{equation*}
\left(\bar{\varphi}^{\delta },\bar{\psi}^{\delta
}\right):=(F_{3}\left( \bar{y}^{\delta },\bar{p}^{\delta
}\right),F_{4}\left( \bar{y}^{\delta },\bar{p}^{\delta }\right)),
\label{3.2}
\end{equation*}

where $\bar{y}^{\delta }$ and $\bar{p}^{\delta }$ belong to
$\mathcal{U}$ can be respectively defined by

\begin{equation}
\bar{y}^{\delta }:=F_{1}\left( \bar{\varphi}^{\delta
},\bar{\psi}^{\delta }\right), \label{3.3}
\end{equation}

and

\begin{equation}
\bar{p}^{\delta }:=F_{3}\left( \bar{y}^{\delta
},\bar{\varphi}^{\delta },\bar{\psi}^{\delta }\right). \label{3.4}
\end{equation}

Then, by the definitions of the mappings $F_{1}, F_{2}, F_{3}$ and
$F_{4}$, the relations \eqref{3.2}, \eqref{3.3} and \eqref{3.4}
are respectively written as

\begin{equation}
A\bar{y}^{\delta }+\beta _{\delta }\left( \bar{y}^{\delta }-\bar{\varphi}%
^{\delta }\right)-\beta _{\delta }\left( \bar{\psi}%
^{\delta }-\bar{y}^{\delta }\right) =f, \text{ in } \Omega, \text{
and } {\bar{y}^{\delta}}=0 \text{ on } \partial\Omega \label{3.5}
\end{equation}

\begin{equation}
A\bar{p}^{\delta }+\beta _{\delta }^{\prime }\left( \bar{y}%
^{\delta }-\bar{\varphi}^{\delta }\right) \bar{p}^{\delta }+\beta
_{\delta }^{\prime }\left( \bar{\psi}^{\delta }-\bar{y} ^{\delta
}\right) \bar{p}^{\delta }=\bar{y}^{\delta }-z, \text{ in }
\Omega, \text{ and } {\bar{p}^{\delta}}=0 \text{ on }
\partial\Omega \label{3.6}
\end{equation}

\begin{equation}
\nu \Delta \bar{\varphi}^{\delta }+\beta _{\delta }^{\prime
}\left( \bar{y}^{\delta }-\bar{\varphi}^{\delta }\right)
\bar{p}^{\delta }=-\bar{\lambda}^{\delta}, \text{ in } \Omega,
\text{ and } {\bar{\varphi}^{\delta}}=0 \text{ on } \partial\Omega
\label{3.7}
\end{equation}

and

\begin{equation}
\nu \Delta \bar{\psi}^{\delta }+\beta _{\delta }^{\prime }\left(
\bar{\psi}^{\delta }-\bar{y}^{\delta }\right) \bar{p}^{\delta }
-\bar{\lambda}^{\delta}=0, \text{ in } \Omega, \text{ and }
{\bar{\psi}^{\delta}}=0 \text{ on }
\partial\Omega \label{3..7}
\end{equation}

 Hence, we remark that the set of equations \eqref{3.5}, \eqref{3.6}, \eqref{3.7} and \eqref{3..7} is the same set of the
equations of the optimality system $({S_\delta})$ when $\left(
y^{\delta },\varphi ^{\delta },\psi ^{\delta },p^{\delta }\right)
$ is replaced by $\left( \bar{y}^{\delta }, \bar{\varphi}^{\delta
},\bar{\psi}^{\delta },\bar{p}^{\delta }\right)$.
\end{proof}

The equations \eqref{2.2}, \eqref{2..8} and \eqref{2..10} of the
optimality system $ \mathcal(S^{\delta })$ are respectively
nonlinear according to $y^{\delta }$, $ \varphi^{\delta }$ and $
\psi^{\delta }$. Therefore for the solution of the system
$(\mathcal{S}^{\delta })$, we propose the following iterative
algorithm.

\begin{algorithm}[h]
\label{Alg2}%
\caption{Newton dumped-Gauss-Seidel algorithm (Continuous
version)}
\begin{algorithmic}[1]

\STATE
 \textbf{Input :}$\left\{ y_{0}^{\delta }, p_{0}^{\delta }, \varphi _{0}^{\delta }, \lambda _{0}^{\delta },\psi _{0}^{\delta },\delta ,\nu
,\omega_{y},\omega_{\varphi},\omega_{\psi},\varepsilon\right\} $
choose $\varphi _{0}^{\delta },\psi _{0}^{\delta }\in
\mathcal{W},\varepsilon $ and $\delta $ in $ \mathbb{R}_{+}^{\ast
};$


\STATE\textbf{Begin:}\\

\STATE \textbf{Calculate } $J_{n-1} \leftarrow
J_{n-1}\left(y^{\delta}_{n-1}, \varphi^{\delta}_{n-1},
\psi^{\delta}_{n-1}\right)$ \STATE \textbf{Step 1}
 \STATE
\textbf{If} $\left( A+\beta _{\delta }^{\prime }\left(
y_{n-1}^{\delta }-\varphi _{n-1}^{\delta }\right)+\beta _{\delta
}^{\prime }\left(\psi _{n-1}^{\delta }- y_{n-1}^{\delta }\right)
\right) $ is singular \textbf{Stop.}

\STATE  \textbf{ Else} \STATE \textbf{\ Solve} $\left( A+\beta
_{\delta }^{\prime }\left( y_{n-1}^{\delta }-\varphi
_{n-1}^{\delta }\right)+\beta _{\delta }^{\prime }\left(\psi
_{n-1}^{\delta }- y_{n-1}^{\delta }\right) \right).r_{n}^{\delta
}=$ \\$ -\omega_{y}\left( A y_{n-1}^{\delta }+\beta _{\delta
}\left( y_{n-1}^{\delta }-\varphi _{n-1}^{\delta } \right)-\beta
_{\delta }\left( \psi _{n-1}^{\delta } -y_{n-1}^{\delta }\right)
-f
\right)$ on $r_{n}^{\delta },$\\

\STATE \ \ \ \ \ \ \ \ \ \ \ \textbf{Then } $y_{n}^{\delta
}=y_{n-1}^{\delta }+$ $r_{n}^{\delta }.$ \\

 \STATE \textbf{End if}

\STATE \textbf{Step 2}
 \STATE \textbf{If} $\left( A+\beta _{\delta
}^{\prime }\left( y_{n}^{\delta }-\varphi _{n-1}^{\delta }\right)
+\beta _{\delta }^{\prime }\left( \psi _{n-1}^{\delta }
-y_{n}^{\delta }\right)\right)$ is singular \textbf{Stop.} \\

\STATE \ \ \ \ \ \ \ \ \  \textbf{\ Else} \\

\STATE \textbf{\ Solve } $\left( A +\beta _{\delta }^{\prime
}\left( y_{n}^{\delta }-\varphi _{n-1}^{\delta }\right)+\beta
_{\delta }^{\prime }\left( \psi _{n-1}^{\delta } -y_{n}^{\delta
}\right) \right) p_{n}^{\delta }=y_{n}^{\delta }-z$
on $p_{n}^{\delta }.$\\

 \STATE \textbf{End if}
 \STATE \textbf{Step 3}
\STATE \textbf{Calculate} $\lambda_{n}^{\delta } = \nu \Delta
\varphi _{n-1}^{\delta }+\beta _{\delta }^{\prime }\left(
y_{n}^{\delta }-\varphi _{n-1}^{\delta }\right) p_{n}^{\delta } $
\STATE \textbf{Step 4}
 \STATE \textbf{If} $\left( \nu \Delta+\beta
_{\delta }^{\prime \prime }\left( \psi _{n-1}^{\delta
}- y_{n}^{\delta }\right) p_{n}^{\delta }\right) $ is not invertible \textbf{Stop.}\\

\STATE \ \ \ \ \ \ \ \ \  \textbf{Else} \\

\STATE \textbf{Solve }$\left(  \nu \Delta+\beta _{\delta }^{\prime
\prime }\left(\psi _{n-1}^{\delta } -y_{n}^{\delta }\right)
 p_{n}^{\delta }\right) .$ $r_{n}^{\delta
}=-\omega_{\psi}\left(\nu A_{h}^{d}\psi _{n-1}^{\delta }+\beta
_{\delta }^{\prime }\left( \psi _{n-1}^{\delta }-y_{n}^{\delta
}\right) p_{n}^{\delta }+\lambda_{n}^{\delta }\right)$ on
$r_{n}^{\delta }.$

\STATE \ \ \ \ \ \ \ \ \ \ \ \textbf{Then }$ \psi _{n}^{\delta
}=\psi _{n-1}^{\delta }+$ $r_{n}^{\delta }$


\STATE \textbf{Step 5}
 \STATE \textbf{If} $\left( \nu \Delta-\beta
_{\delta }^{\prime \prime }\left(y_{n}^{\delta }- \varphi
_{n-1}^{\delta
}\right) p_{n}^{\delta }\right) $ is not invertible \textbf{Stop.}\\

\STATE \ \ \ \ \ \ \ \ \  \textbf{Else} \\

\STATE \textbf{Solve }$\left(  \nu \Delta-\beta _{\delta }^{\prime
\prime }\left( y_{n}^{\delta }-\varphi _{n-1}^{\delta } \right)
p_{n}^{\delta }\right) .$ $r_{n}^{\delta
}=-\omega_{\varphi}\left(\nu A_{h}^{d}\varphi _{n-1}^{\delta
}+\beta _{\delta }^{\prime }\left( y_{n}^{\delta }-\varphi
_{n-1}^{\delta }\right) p_{n}^{\delta }-\lambda_{n}^{\delta
}\right)$ on $r_{n}^{\delta }.$

\STATE \ \ \ \ \ \ \ \ \ \ \ \textbf{Then }$ \varphi _{n}^{\delta
}=\varphi _{n-1}^{\delta }+$ $r_{n}^{\delta }$ \STATE
\textbf{Calculate } $J_{n} \leftarrow J_{n-1}\left(y^{\delta}_{n},
\varphi^{\delta}_{n},\psi^{\delta}_{n}\right)$



 \STATE \textbf{End if}


 \STATE \textbf{If } $|J_{n}-J_{n-1}| \leq  \varepsilon $ \textbf{Stop.}

\STATE \textbf{Ensure :} $s_{n}^{\delta }:=\left( y_{n}^{\delta
},\varphi_{n}^{\delta },\psi_{n}^{\delta },p_{n}^{\delta }\right)
$ \textbf{is a solution}

\STATE \ \ \ \ \ \ \ \ \  \textbf{Else; } $n\leftarrow n+1$,
\textbf{Go to} \textbf{Begin.}

 \STATE \textbf{End if}

\STATE \textbf{End algorithm}.
\end{algorithmic}
\end{algorithm}
\FloatBarrier

\subsection{Convergence results}

In this subsection, we give some conditions on $\delta$ and
$\omega$ to have the convergence of the above algorithm. We denote
by $\bar{y}^{\delta }$, $\bar{p}^{\delta }$,
$\bar{\varphi}^{\delta }$ and $\bar{\psi}^{\delta }$ the solutions
of the equations \eqref{3.5}, \eqref{3.6}, \eqref{3.7} and
\eqref{3..7} respectively, and let $y_{n}^{\delta }$,
$\lambda_{n}^{\delta }$, $p_{n}^{\delta }$, $\psi_{n}^{\delta }$
and $\varphi_{n}^{\delta }$ be given respectively  by step 1, step
2, step 3, step 4, step 5 respectively of the latter algorithm.

\begin{remark}  From Lemma \ref{lm:l2}, if we replace
$y_{n}^{\delta ,2},\varphi _{n-1}^{\delta,2 },\psi
_{n-1}^{\delta,2 }$ and $p_{n}^{\delta,2 }$ respectively by
$\bar{y}^{\delta },\bar{\varphi} ^{\delta },\bar{\psi} ^{\delta }$
and $\bar{p}^{\delta }$, we get

\begin{equation*}
\parallel p_{n}^{\delta,1 }-\bar{p}^{\delta }\parallel _{H^{1}\left( \Omega
\right) }\leq l_{2}\left( \parallel y_{n}^{\delta,1
}-\bar{y}^{\delta }\parallel _{L^{2}\left( \Omega \right)
}+\parallel \varphi _{n-1}^{\delta,1 }- \bar{\varphi}^{\delta
}\parallel _{L^{2}\left( \Omega \right) }+\parallel \psi
_{n-1}^{\delta,1 }- \bar{\psi}^{\delta }\parallel _{L^{2}\left(
\Omega \right) }\right), \label{3.8}
\end{equation*}

where $l_{2}=C+\tfrac{C\tilde{\rho_{3}}}{\delta}$.

\end{remark}

\begin{lemma}\label{lem:ll}
Let $\bar{\lambda}^{\delta }$ in $\mathcal{U}$ be the solution of
the following equation
\begin{equation*}
\bar{\lambda}^{\delta } = \nu \Delta \bar{\varphi}^{\delta }+\beta
_{\delta }^{\prime }\left( \bar{y}^{\delta }-\bar{\varphi}^{\delta
}\right) \bar{p}^{\delta },
\end{equation*}

since

\begin{equation*}
\parallel \bar{p}^{\delta }\parallel _{H^{1}\left( \Omega \right) }\leq
\tilde{\rho}_{3} \label{3.9}
\end{equation*}

we obtain

\begin{equation*}
\parallel \lambda _{n}^{\delta }-\bar{\lambda}^{\delta }\parallel
_{L^{2}\left( \Omega \right) }\leq k_{\lambda}\left( \parallel y_{n}^{\delta }-%
\bar{y}^{\delta }\parallel _{L^{2}\left( \Omega \right)
}+\parallel \varphi _{n-1}^{\delta }-\bar{\varphi}^{\delta
}\parallel _{L^{2}\left( \Omega \right) }+\parallel p_{n}^{\delta
}-\bar{p}^{\delta }\parallel _{L^{2}\left( \Omega \right) }
\right) \label{3.11a}
\end{equation*}

where $ k_{\lambda}:=\dfrac{C}{\delta}$ and $\tilde{\rho}_{3}\leq
C\delta \nu $.
\end{lemma}
\begin{proof}
From step $3$ of the continuous version of the algorithm $2$, and
by Lemma \ref{lm:b1}, we get
\begin{multline*}
\parallel
\lambda_{n}^{\delta }-\bar{\lambda}^{\delta }\parallel
_{L^{2}\left( \Omega \right) } \leq (\nu
+\dfrac{C\tilde{\rho}_{3}}{\delta})\parallel  \varphi
_{n-1}^{\delta }-\bar{\varphi}^{\delta }\parallel _{H^{2}\left(
\Omega \right) }+\dfrac{C}{\delta}\parallel p_{n}^{\delta }-
\bar{p}^{\delta }\parallel _{H^{1}\left( \Omega \right)
}+\\
 +\parallel y_{n}^{\delta
}-\bar{y}^{\delta }\parallel _{L^{2}\left( \Omega \right) },
\end{multline*}

then, we get

\begin{multline*}
\parallel
\lambda_{n}^{\delta }-\bar{\lambda}^{\delta }\parallel
_{L^{2}\left( \Omega \right) } \leq k_{\lambda}(\parallel  \varphi
_{n-1}^{\delta }-\bar{\varphi}^{\delta }\parallel _{H^{2}\left(
\Omega \right) }+\parallel p_{n}^{\delta }- \bar{p}^{\delta
}\parallel _{H^{1}\left( \Omega \right)
}+\\
 \dfrac{C\tilde{\rho}_{3}}{\delta}\parallel y_{n}^{\delta
}-\bar{y}^{\delta }\parallel _{L^{2}\left( \Omega \right) }),
\end{multline*}

where

\begin{equation*}
k_{\lambda}:=max\{(\nu
+\dfrac{C\tilde{\rho}_{3}}{\delta}),\dfrac{C}{\delta},\dfrac{C\tilde{\rho}_{3}}{\delta}\}=\dfrac{C}{\delta}.
\end{equation*}
\end{proof}
\begin{lemma}\label{lem:ll}
Let $\bar{\varphi}^{\delta }$ in $\mathcal{U}$ be the solution of
$\left( \ref{3.7}\right)$, since

\begin{equation*}
\parallel \bar{p}^{\delta }\parallel _{H^{1}\left( \Omega \right) }\leq
\tilde{\rho}_{3},  \label{3.9}
\end{equation*}

where $ \omega_{\varphi }$ is strictly positive, such that

\begin{equation*}
\dfrac{\delta \nu C+C\tilde{\rho}_{3}}{\left( C+\delta \nu
C-C\tilde{\rho}_{3}\right) }\leq \omega_{\varphi }\leq 1,
\label{3.10}
\end{equation*}

and
\begin{equation*}
\omega _{\varphi }<\dfrac{\delta^2 \nu
C-C\delta\tilde{\rho}_{3}}{C\delta+C\tilde{\rho}_{3}},
\end{equation*}

we obtain

\begin{multline}
\parallel \varphi _{n}^{\delta }-\bar{\varphi}^{\delta }\parallel
_{H^{2}\left( \Omega \right) }\leq k_{3}( \parallel y_{n}^{\delta }-%
\bar{y}^{\delta }\parallel _{H^{1}\left( \Omega \right)
}+\parallel \varphi _{n-1}^{\delta }-\bar{\varphi}^{\delta
}\parallel _{H^{2}\left( \Omega \right) }+\\
+\parallel p_{n}^{\delta }-\bar{p}^{\delta }\parallel
_{H^{1}\left( \Omega \right) } +
\parallel \lambda _{n}^{\delta
}-\bar{\lambda}^{\delta }\parallel _{{H^{1}\left( \Omega \right)
}}), \label{3.11}
\end{multline}

where $ k_{3}:=\omega_{\varphi }\dfrac{C}{\delta \nu
C-C\tilde{\rho}_{3}}$ and $\tilde{\rho}_{3}\leq C\delta \nu $.
\end{lemma}

\begin{proof}
From step 5 of the continuous version of the algorithm 2, by the continuity and coercivity conditions H\textsubscript{1} and
H\textsubscript{2} of $\sigma\left( \cdot , \cdot \right)$, we obtain

\begin{equation*}
\begin{split}
\left( \tfrac{\delta \nu C-C\tilde{\rho}_{3}}{\delta }\right)
\parallel \varphi _{n}^{\delta }-\bar{\varphi}^{\delta }\parallel
_{{H^{2}\left(
\Omega \right) }}&\leq \left( \tfrac{\left( 1-\omega_{\varphi }\right)\delta\nu C+%
\left( 1+\omega_{\varphi }\right)C\tilde{\rho}_{3}}{\delta
}\right) \parallel \varphi _{n-1}^{\delta }-\bar{\varphi}^{\delta
}\parallel
_{{H^{2}\left( \Omega \right) }}+\\
 &+\omega_{\varphi
}\tfrac{C\tilde{\rho}_{3}}{\delta }\parallel y_{n}^{\delta
}-\bar{y}^{\delta }\parallel _{{H^{1}\left( \Omega \right)
}}+\omega_{\varphi }\tfrac{C}{\delta }\parallel p_{n}^{\delta }+\bar{p%
}^{\delta }\parallel _{{H^{1}\left( \Omega \right) }}+\\
& + \omega_{\varphi } C\parallel \lambda _{n}^{\delta
}-\bar{\lambda}^{\delta }\parallel _{{H^{1}\left( \Omega \right)
}}.
\end{split}
\end{equation*}

Finally, we obtain

\begin{equation*}
\begin{split}
\parallel \varphi _{n}^{\delta }-\bar{\varphi}^{\delta }\parallel
_{{H^{2}\left( \Omega \right) }}&\leq \left( \dfrac{\left( 1-\omega_{\varphi }\right)\delta\nu C+%
\left( 1+\omega_{\varphi }\right)C\tilde{\rho}_{3}}{\delta \nu
C-C\tilde{\rho}_{3}
}\right) \parallel \varphi _{n-1}^{\delta }-\bar{\varphi}%
^{\delta }\parallel _{{H^{2}\left( \Omega \right) }}+\\
&+\omega_{\varphi }\dfrac{C\tilde{\rho}_{3}}{\delta \nu
C-C\tilde{\rho}_{3}}\parallel y_{n}^{\delta }-\bar{y}^{\delta
}\parallel _{{H^{1}\left(
\Omega \right)}}+\omega_{\varphi }\dfrac{C}{\delta \nu C-C%
\tilde{\rho}_{3}}\parallel p_{n}^{\delta }-\bar{p}^{\delta
}\parallel _{{H^{1}\left( \Omega \right)}}+\\
& +\left( \tfrac{\delta \omega_{\varphi } C }{\delta \nu
C-C\tilde{\rho}_{3}}\right)\parallel \lambda _{n}^{\delta
}-\bar{\lambda}^{\delta }\parallel _{{H^{1}\left( \Omega \right)
}}.
\end{split}
\end{equation*}
\end{proof}

\begin{lemma}\label{lem:llpsi}
Let $\bar{\psi}^{\delta }$ in $\mathcal{U}$ be the solution of
\eqref{3..7}, since

\begin{equation*}
\parallel \bar{p}^{\delta }\parallel _{H^{1}\left( \Omega \right) }\leq
\tilde{\rho}_{3}  \label{3.9}
\end{equation*}

where $ \omega_{\psi }$ is strictly positive, such that

\begin{equation*}
\dfrac{\delta \nu C+C\tilde{\rho}_{3}}{\left( C+\delta \nu
C-C\tilde{\rho}_{3}\right) }\leq \omega_{\psi }\leq 1
\label{3..10}
\end{equation*}

and

\begin{equation*}
\omega _{\psi }<\dfrac{\delta^2 \nu
C-C\delta\tilde{\rho}_{3}}{C\delta+C\tilde{\rho}_{3}}.
\end{equation*}

Then, we obtain

\begin{multline*}
\parallel \psi _{n}^{\delta }-\bar{\psi}^{\delta }\parallel
_{H^{2}\left( \Omega \right) }\leq k_{2}( \parallel y_{n}^{\delta }-%
\bar{y}^{\delta }\parallel _{H^{1}\left( \Omega \right)
}+\parallel \psi _{n-1}^{\delta }-\bar{\psi}^{\delta }\parallel
_{H^{2}\left( \Omega \right) }+\parallel p_{n}^{\delta
}-\bar{p}^{\delta }\parallel _{H^{1}\left( \Omega \right) }+\\
+ \parallel \lambda _{n}^{\delta }-\bar{\lambda}^{\delta
}\parallel _{{H^{1}\left( \Omega \right) }}) \label{3..11}
\end{multline*}

where $\tilde{\rho}_{3}\leq C\delta \nu $ and $
k_{2}:=\omega_{\psi }\dfrac{C}{\delta \nu C-C\tilde{\rho}_{3}}$.
\end{lemma}

\begin{proof}
From step 4 of the continuous version of the algorithm $2$, by the continuity and coercivity conditions H\textsubscript{1} and
H\textsubscript{2} of $\sigma\left( \cdot , \cdot \right)$, we get

\begin{equation*}
\begin{split}
\left( \tfrac{\delta \nu C-C\tilde{\rho}_{3}}{\delta }\right)
\parallel \psi _{n}^{\delta }-\bar{\psi}^{\delta }\parallel
_{{H^{2}\left(
\Omega \right) }}&\leq \left( \tfrac{\left( 1-\omega_{\psi }\right)\delta\nu C+%
\left( 1+\omega_{\psi }\right)C\tilde{\rho}_{3}}{\delta }\right)
\parallel \psi _{n-1}^{\delta }-\bar{\psi}^{\delta
}\parallel
_{{H^{2}\left( \Omega \right) }}+\\
 &+\omega_{\psi
}\tfrac{C\tilde{\rho}_{3}}{\delta }\parallel y_{n}^{\delta
}-\bar{y}^{\delta }\parallel _{{H^{1}\left( \Omega \right)
}}+\omega_{\psi }\tfrac{C}{\delta }\parallel p_{n}^{\delta }-\bar{p%
}^{\delta }\parallel _{{H^{1}\left( \Omega \right) }} +\\
& + \omega_{\psi }C\parallel \lambda _{n}^{\delta
}-\bar{\lambda}^{\delta }\parallel _{{H^{1}\left( \Omega \right)
}} .
\end{split}
\end{equation*}

Finally, we obtain
\begin{equation*}
\begin{split}
\parallel \psi _{n}^{\delta }-\bar{\psi}^{\delta }\parallel
_{{H^{2}\left( \Omega \right) }}&\leq \left( \dfrac{\left( 1-\omega_{\psi }\right)\delta\nu C+%
\left( 1+\omega_{\psi }\right)C\tilde{\rho}_{3}}{\delta \nu
C-C\tilde{\rho}_{3}
}\right) \parallel \psi _{n-1}^{\delta }-\bar{\psi}%
^{\delta }\parallel _{{H^{2}\left( \Omega \right) }}+\\
&+\omega_{\psi }\dfrac{C\tilde{\rho}_{3}}{\delta \nu
C-C\tilde{\rho}_{3}}\parallel y_{n}^{\delta }-\bar{y}^{\delta
}\parallel _{{H^{1}\left(
\Omega \right)}}+\omega_{\psi }\dfrac{C}{\delta \nu C-C%
\tilde{\rho}_{3}}\parallel p_{n}^{\delta }-\bar{p}^{\delta
}\parallel _{{H^{1}\left( \Omega \right)}}+\\
& + \left( \tfrac{\omega_{\psi }C \delta }{\delta \nu
C-C\tilde{\rho}_{3}}\right)\parallel \lambda _{n}^{\delta
}-\bar{\lambda}^{\delta }\parallel _{{H^{1}\left( \Omega \right)
}}.
\end{split}
\end{equation*}
\end{proof}

\begin{lemma}
Let $y_{n}^{\delta }$ in $\mathcal{U}$ be the solution of
\eqref{3.5}, since the condition \eqref{3.9} of previous Lemma
\ref{lem:llpsi} is fulfilled, where

\begin{equation*}
\left( \dfrac{\delta C+C-\delta }{\delta C+C}\right)
<\omega_{y}\leq \dfrac{\left( \delta C+C\right) }{\left( \delta
C+2C\right) } \leq 1,  \label{3.12}
\end{equation*}

we get

\begin{equation*}
\parallel y_{n}^{\delta }-\bar{y}^{\delta }\parallel _{{H^{1}\left(
\Omega \right) }}\leq k_{1}\left\{ \parallel e_{n-1}^{\delta
}-\bar{e}^{\delta }
\parallel _{\mathcal{V}}^{2}+\parallel e_{n-1}^{\delta }-\bar{e}^{\delta }\parallel
_{\mathcal{V}}\right\} \label{3.13}
\end{equation*}

where $k_{1}:=\left( 1-\omega_{y}\right) \left( C+\dfrac{C}{\delta
}\right),$ $e_{n-1}^{\delta }:=\left( y_{n-1}^{\delta },\varphi
_{n-1}^{\delta },\psi _{n-1}^{\delta }\right) $, $\bar{e}^{\delta
}:=\left( \bar{y}^{\delta },\bar{\varphi}^{\delta
},\bar{\psi}^{\delta }\right)$ and $\mathcal{V}:={H^{1}\left(
\Omega \right) \times H^{2}\left( \Omega \right)\times H^{2}\left(
\Omega \right) }$.
\end{lemma}
\begin{proof}

From step 1 of the algorithm 2, and since \newline
 $-\left((\beta _{\delta }^{\prime }\left(
y_{n-1}^{\delta }-\varphi _{n-1}^{\delta }\right)+\beta _{\delta
}^{\prime }\left( \psi _{n-1}^{\delta }-y_{n-1}^{\delta }\right))
\left( y_{n}^{\delta }-\bar{y}^{\delta }\right),\left(
y_{n}^{\delta }-\bar{y}^{\delta }\right) \right)$$\leq0$, by the coercivity and continuity
conditions H\textsubscript{1} and H\textsubscript{2} of
$\sigma\left(\cdot, \cdot\right)$, we obtain

\begin{multline*}
\parallel y_{n}^{\delta }-\bar{y}^{\delta }\parallel _{{H^{1}\left(
\Omega \right) }}  \leq k_{1}\left\{ \parallel y_{n-1}^{\delta
}-\bar{y}^{\delta }
\parallel _{{H^{1}\left( \Omega \right) }}^{2}+\parallel \varphi
_{n-1}^{\delta }-\bar{\varphi}^{\delta }\parallel _{{H^{2}\left(
\Omega \right) }}^{2}+\parallel \psi _{n-1}^{\delta
}-\bar{\psi}^{\delta }\parallel _{{H^{2}\left( \Omega \right)
}}^{2}\right. \\  \left.+\parallel y_{n-1}^{\delta
}-\bar{y}^{\delta }\parallel
_{{H^{1}\left( \Omega \right) }}+\parallel \varphi _{n-1}^{\delta }-\bar{%
\varphi}^{\delta }\parallel _{{H^{2}\left( \Omega \right)
}}+\parallel \psi _{n-1}^{\delta }-\bar{\psi}^{\delta }\parallel
_{{H^{2}\left( \Omega \right) }}\right\},
\end{multline*}
where
\begin{equation*}
k_{1}:=\max \left\{ \left( 1-\omega_{y}\right) \left( C+\dfrac{1}{\delta }%
C\right) ,\omega_{y}\dfrac{1}{\delta }C\left( 1-\theta \right) ,\omega_{y}%
\dfrac{1}{\delta }C\right\}=\left( 1-\omega_{y}\right) \left( C+\dfrac{1}{\delta }%
C\right). \end{equation*}
\end{proof}

\begin{theorem} \label{thr2}
Let $e_{n}^{\delta }:=\left( y_{n}^{\delta },\varphi _{n}^{\delta
},\psi _{n}^{\delta }\right) $, $\bar{e}^{\delta }:=\left(
\bar{y}^{\delta },\bar{\varphi}^{\delta },\bar{\psi}^{\delta
}\right) $ and $\mathcal{V}:={H^{1}\left( \Omega \right) \times
H^{2}\left( \Omega \right)\times H^{2}\left( \Omega \right) }$,
then we get
\begin{equation*}
\parallel e_{n}^{\delta }-\bar{e}^{\delta }\parallel _{\mathcal{V}}\leq k\max \left\{ \parallel e_{n-1}^{\delta
}-\bar{e}^{\delta }\parallel _{\mathcal{V}}^{2},\parallel
e_{n-1}^{\delta }-\bar{e}^{\delta }\parallel
_{\mathcal{V}}\right\}, \label{3.14}
\end{equation*}

where

\begin{equation*}
k:=2\left( k_{1}+\tilde{k}_{3}\right),\text{
}\tilde{k}_{3}:=k_{3}\left( k_{1}+\tilde{l}_{2}+1\right)\text{ and
}\tilde{l}_{2}=l_{2}\left( Ck_{1}+C\right) .
\end{equation*}
\end{theorem}

\begin{proof}
From equations \eqref{3.13} and \eqref{3.8}, we get

\begin{equation*}
\parallel y_{n}^{\delta }-\bar{y}^{\delta }\parallel _{\mathcal{V}}\leq k_{1}\left\{ \parallel e_{n-1}^{\delta }-\bar{e}^{\delta }%
\parallel _{\mathcal{V}}^{2}+\parallel e_{n-1}^{\delta }-\bar{e}^{\delta }\parallel
_{\mathcal{V}}\right\},
\end{equation*}

and
\begin{multline*}
\parallel p_{n}^{\delta }-\bar{p}^{\delta }\parallel _{H^{1}\left( \Omega
\right) }\leq l_{2}( C\parallel y_{n}^{\delta }-\bar{y}^{\delta
}\parallel _{H^{1}\left( \Omega \right)} +C\parallel \varphi
_{n-1}^{\delta }-\bar{\varphi}^{\delta
}\parallel _{H^{2}\left( \Omega \right) }\\
+C\parallel \psi _{n-1}^{\delta }-\bar{\psi}^{\delta }\parallel
_{H^{2}\left( \Omega \right) }),
\end{multline*}
then we obtain
\begin{equation*}
\parallel p_{n}^{\delta }-\bar{p}^{\delta }\parallel _{H^{1}\left( \Omega
\right) }\leq \tilde{l}_{2}\left\{ \parallel e_{n-1}^{\delta }-\bar{e}^{\delta }%
\parallel _{\mathcal{V}}^{2}+\parallel e_{n-1}^{\delta }-\bar{e}^{\delta }\parallel
_{\mathcal{V}}\right\}, \label{3.15}
\end{equation*}

where
\begin{equation*}
\tilde{l}_{2}:=l_{2}\left( Ck_{1}+C\right).
\end{equation*}
And by equation \eqref{3.11}, we get
\begin{multline}
\parallel
\lambda_{n}^{\delta }-\bar{\lambda}^{\delta }\parallel
_{L^{2}\left( \Omega \right) } \leq k_{\lambda}(\parallel  \varphi
_{n-1}^{\delta }-\bar{\varphi}^{\delta }\parallel _{H^{2}\left(
\Omega \right) }+\tilde{l}_{2}\left\{ \parallel e_{n-1}^{\delta }-\bar{e}^{\delta }%
\parallel _{\mathcal{V}}^{2}+\parallel e_{n-1}^{\delta }-\bar{e}^{\delta }\parallel
_{\mathcal{V} }\right\}+\\
 \dfrac{C\tilde{\rho}_{3}}{\delta}k_{1}\left\{ \parallel e_{n-1}^{\delta }-\bar{e}^{\delta }%
\parallel _{\mathcal{V}}^{2}+\parallel e_{n-1}^{\delta }-\bar{e}^{\delta }\parallel
_{\mathcal{V} }\right\})
\end{multline}
then, we obtain
\begin{equation*}
\parallel \varphi _{n}^{\delta }-\bar{\varphi}^{\delta }\parallel
_{H^{2}\left( \Omega \right) }\leq \tilde{k}_{3}\left\{
\parallel e_{n-1}^{\delta }-\bar{e}^{\delta }\parallel _{\mathcal{V}}^{2}+\parallel e_{n-1}^{\delta }-\bar{e}^{\delta }%
\parallel _{\mathcal{V}}\right\}  \label{3.16}
\end{equation*}

where
\begin{equation*}
\tilde{k}_{3}:=k_{3}\left(
k_{1}+\tilde{l}_{2}+1+\tilde{k}_{\lambda}\right),
\end{equation*}

and by equation \eqref{3.11}, we get

\begin{multline*}
\parallel \psi _{n}^{\delta }-\bar{\psi}^{\delta }\parallel
_{H^{2}\left( \Omega \right) }\leq  k_{2}( \parallel y_{n}^{\delta }-%
\bar{y}^{\delta }\parallel _{H^{1}\left( \Omega \right)
}+\parallel \psi _{n-1}^{\delta }-\bar{\psi}^{\delta }\parallel
_{H^{2}\left( \Omega \right) }+\parallel p_{n}^{\delta
}-\bar{p}^{\delta }\parallel _{H^{1}\left( \Omega \right)
} + \\
 +\parallel \lambda_{n}^{\delta }-\bar{\lambda}^{\delta
}\parallel _{H^{1}\left( \Omega \right) }),
\end{multline*}

 then, we obtain

\begin{equation}
\parallel \psi _{n}^{\delta }-\bar{\psi}^{\delta }\parallel
_{H^{2}\left( \Omega \right) }\leq \tilde{k}_{2}\left\{
\parallel e_{n-1}^{\delta }-\bar{e}^{\delta }\parallel _{\mathcal{V}}^{2}+\parallel e_{n-1}^{\delta }-\bar{e}^{\delta }%
\parallel _{\mathcal{V}}\right\}  \label{3..16}
\end{equation}

where
\begin{equation*}
\tilde{k}_{2}:=k_{2}\left(
k_{1}+\tilde{l}_{2}+1+\tilde{k}_{\lambda}\right).
\end{equation*}

From equations \eqref{3.13}, \eqref{3.16} and \eqref{3..16}, we
get

\begin{equation*}
\parallel e_{n}^{\delta }-\bar{e}^{\delta }\parallel _{\mathcal{V}}\leq 2\left(
k_{1}+\tilde{k}_{3}+\tilde{k}_{2}\right) \max \left\{
\parallel e_{n-1}^{\delta }-\bar{e}^{\delta }\parallel _{\mathcal{V}}^{2},\parallel e_{n-1}^{\delta }-\bar{e}^{\delta }%
\parallel _{\mathcal{V}}\right\}.
\end{equation*}
Finally, we get
\begin{equation}
\parallel e_{n}^{\delta }-\bar{e}^{\delta }\parallel _{\mathcal{V}}\leq k\max \left\{ \parallel e_{n-1}^{\delta
}-\bar{e}^{\delta }\parallel _{\mathcal{V}}^{2},\parallel
e_{n-1}^{\delta }-\bar{e}^{\delta }\parallel
_{\mathcal{V}}\right\} \label{3.17}
\end{equation}
where
\begin{equation*}
k:=2\left( k_{1}+\tilde{k}_{3}+\tilde{k}_{2}\right).
\end{equation*}
\end{proof}

\begin{remark}
As seen above, it is very difficult to give a sharp estimate of
the constant $k$ and to prove that this constant is less than $1$
to get the convergence of the latter algorithm. However, we
believe that with suitable choices of $\delta$ and $\omega$, we
can make this constant less than $1$. \end{remark}

\begin{remark}
\label{Rem_1} From Theorem \ref{thr2}, we deduce that
$y^{\delta}_{n}$ converges strongly to $\bar{y}^{\delta}$ in
$\Huz$ and $\varphi^{\delta}_{n}$ converges strongly to
$\bar{\varphi}^{\delta}$ in $\Hd$ and $\psi^{\delta}_{n}$
converges strongly to $\bar{\psi}^{\delta}$ in $\Hd$.
\end{remark}

\begin{corollary}
By the assumptions of Theorem \ref{thr2}, we deduce that
\begin{equation*}
\mid J\left(y_{n}^{\delta },\varphi_{n}^{\delta },\psi_{n}^{\delta
}\right)-J\left(y_{n-1}^{\delta },\varphi_{n-1}^{\delta
},\psi_{n-1}^{\delta }\right)\mid \text{ goes to } 0.
\end{equation*}
\end{corollary}
\begin{proof}
From the cost functional defined in $(\mathcal{P}^{\delta})$, we
can write
\begin{multline*}
\mid J\left(y_{n}^{\delta },\varphi_{n}^{\delta },\psi_{n}^{\delta
}\right)-J\left(y_{n-1}^{\delta },\varphi_{n-1}^{\delta
},\psi_{n-1}^{\delta }\right)\mid =
\dfrac{1}{2}\mid\int\nolimits_{\Omega }\left( y_{n}^{\delta }
-z\right) ^{2}d{ x}+\\
\nu\left( \int\nolimits_{\Omega }\left( \nabla \varphi_{n}^{\delta
} \right) ^{2}+\left( \nabla \psi_{n}^{\delta } \right)
^{2}d{x}\right) -\\
\left(\int\nolimits_{\Omega }\left( y_{n-1}^{\delta } -z\right)
^{2}d{ x}+\nu\left( \int\nolimits_{\Omega }\left( \nabla
\varphi_{n-1}^{\delta } \right) ^{2}+\left( \nabla
\psi_{n-1}^{\delta } \right) ^{2}d{x}\right)\right)\mid.
\end{multline*}

From Corollary $\ref{cor:c1}$, we have $\parallel y_{n-1}^{\delta
}\parallel_{L^{2}\left( \Omega\right)}\leq\tilde{\rho}_{2}$,
$\parallel \varphi_{n-1}^{\delta }\parallel_{H^{2}\left(
\Omega\right)}\leq\tilde{\rho}_{1}$ and $
\parallel \psi_{n-1}^{\delta }\parallel_{H^{2}\left(
\Omega\right)}\leq\tilde{\rho}_{1}$, then, we deduce that

\begin{multline*}
\mid  J\left(y_{n}^{\delta },\varphi_{n}^{\delta },
\psi_{n}^{\delta }\right)-J\left(y_{n-1}^{\delta
},\varphi_{n-1}^{\delta }, \psi_{n-1}^{\delta }\right)\mid\leq
\dfrac{1}{2}(
\parallel y_{n}^{\delta } - y_{n-1}^{\delta }\parallel^{2}_{L^{2}\left( \Omega\right)}
+( 2\tilde{\rho}_{2}+C)\parallel y_{n}^{\delta }-y_{n-1}^{\delta
}\parallel_{L^{2}\left(
\Omega\right)}\\
+\nu \left(\parallel \nabla \varphi_{n}^{\delta }- \nabla
\varphi_{n-1}^{\delta }\parallel^{2}_{L^{2}\left(
\Omega\right)}+(2\tilde{\rho}_{1})\parallel \nabla
\varphi_{n}^{\delta }-\nabla \varphi_{n-1}^{\delta
}\parallel_{L^{2}\left(
\Omega\right)} \right)\\
 +\nu \left(\parallel \nabla \psi_{n}^{\delta }- \nabla
\psi_{n-1}^{\delta }\parallel^{2}_{L^{2}\left(
\Omega\right)}+(2\tilde{\rho}_{1})\parallel \nabla
\psi_{n}^{\delta }-\nabla \psi_{n-1}^{\delta
}\parallel_{L^{2}\left( \Omega\right)} \right)).
\end{multline*}

Finally, we deduce that $ \mid J\left(y_{n}^{\delta
},\varphi_{n}^{\delta },\psi_{n}^{\delta
}\right)-J\left(y_{n-1}^{\delta },\varphi_{n-1}^{\delta
},\psi_{n-1}^{\delta }\right)\mid$  strongly converges to 0.
\end{proof}

\section{Numerical implementation and computational aspects}

Numerical experiments are carried out for one and two dimensional
problems.  We will attempt to compute a grid function consisting
of values ${y}^{\delta,h }\mathrel{\mathop:}=\left( y_{0}^{\delta
},y_{1}^{\delta },...,y_{N+1}^{\delta }\right),
\\{\varphi} ^{\delta,h }\mathrel{\mathop:}=\left( \varphi
_{0}^{\delta },\varphi _{1}^{\delta },...,\varphi _{N+1}^{\delta
}\right), $ ${\psi} ^{\delta,h }\mathrel{\mathop:}=\left( \psi
_{0}^{\delta },\psi _{1}^{\delta },...,\psi _{N+1}^{\delta
}\right) $ and $ {p}^{\delta ,h }\mathrel{\mathop:}=\left(
p_{0}^{\delta },p_{1}^{\delta },...,p_{N+1}^{\delta }\right), $
where ${y}^{\delta,h },{\varphi} ^{\delta,h }, {\psi} ^{\delta,h
}$ and $ {p}^{\delta ,h}$ are the vectors values of the discrete
solutions of the optimality system $(\mathcal{S}^{\delta})$ such
that $y_{i}^{\delta }\mathrel{\mathop:}=y^{\delta }\left(
x_{i}\right) ,\varphi _{i}^{\delta }\mathrel{\mathop:}=\varphi
^{\delta }\left( x_{i}\right) ,\psi _{i}^{\delta
}\mathrel{\mathop:}=\psi ^{\delta }\left( x_{i}\right) $ and $
p_{i}^{\delta }\mathrel{\mathop:}=p^{\delta }\left( x_{i}\right)$
for $ 0\leq i\leq N+1$, finite-differences approximations
involving the three,
 respectively five, point approximation of the Laplacian in one
 dimensional space, respectively two dimensional space. Here
$x_{i}=ih$ for $ 0\leq i\leq N+1$ and
$h\mathrel{\mathop:}=\dfrac{1}{N+1}$ is the distance between two
successive grid points. From the boundary conditions
$y_{0}^{\delta }=y_{N+1}^{\delta }=0,p_{0}^{\delta
}=p_{N+1}^{\delta }=0,$ $\varphi _{0}^{\delta }=\varphi
_{N+1}^{\delta }=0,$ and $\psi _{0}^{\delta }=\psi _{N+1}^{\delta
}=0,$ so we have $4N$ unknown values to compute in one dimensional
space. Then, for example, if we replace $y^{\left( 2\right)
}\left( x\right) $ (respectively $\Delta y\left( x\right) )$ by
the centered difference approximation, we get
\begin{equation}\label{eqI1}
-y^{\left( 2\right) }\left( x\right):=\tfrac{1}{h^2}(
-y_{i+1}+2y_{i}-y_{i-1}), \text{ where } 0\leq i\leq N+1,
\end{equation}
 and
respectively
\begin{equation}\label{eqI2}
-\left( \Delta y\right) _{ij}:=\tfrac{1}{h^{2}}\left(
-y_{i+1,j}+4y_{i,j}-y_{i-1,j}-y_{i,j+1}-y_{i,j-1}\right) , \text{
where } 0\leq i,j\leq N+1.
\end{equation}
Then, we can write the previous systems under the matrix form, as
\begin{equation*}
\left\{
\begin{array}{l}
A_{h}^{d}y _{n}^{\delta,h }+\beta_{\delta }({y_{n}^{\delta,h
}-\varphi _{n-1}^{\delta,h }})-\beta_{\delta }({\psi
_{n-1}^{\delta,h }-y_{n}^{\delta,h
}})=f^{h},\\
\\
 (A_{h}^{d}+\beta_{\delta
}^{\prime }({y_{n}^{\delta,h }-\varphi _{n-1}^{\delta,h
}})+\beta_{\delta }^{\prime }(\psi
_{n-1}^{\delta,h }-{y_{n}^{\delta,h }}))p_{n}^{\delta,h }=y_{n}^{\delta,h }-z^{h},\\
\\
\lambda_{n}^{\delta,h }=\nu A_{h}^{d}\varphi _{n-1}^{\delta
,h}+\beta_{\delta} ^{\prime }\left( y_{n}^{\delta,h }-\varphi
_{n-1}^{\delta,h }\right)
p_{n}^{\delta,h },\\
\\
\nu A_{h}^{d}\psi _{n}^{\delta ,h}+\beta_{\delta} ^{\prime
}\left(\psi _{n}^{\delta,h } -y_{n}^{\delta,h }\right)
p_{n}^{\delta,h }=-\lambda_{n}^{\delta,h },\\
\\
\nu A_{h}^{d}\varphi _{n}^{\delta ,h}+\beta_{\delta} ^{\prime
}\left(y_{n}^{\delta,h }-\varphi _{n}^{\delta,h } \right)
p_{n}^{\delta,h }-\lambda_{n}^{\delta,h }=0,\\
\end{array}%
\right.
\end{equation*}

where $d=1,2$, $ {f^{h}}:=(f_{0},f_{1},...,f_{N+1} )$,
${z^{h}}:=(z_{0},z_{1},...,z_{N+1} )$, and
 such that for one dimensional problem, $A_{h}^{1}$ is
$(N+2)\times(N+2)$ symmetric positive definite matrix, where
$A_{h}^{1}$ is given in \eqref{eqI1} and for two dimensional
problem $A_{h}^{2}$ is $(N+2)^{2}\times(N+2)^{2}$ symmetric
matrix, where $A_{h}^{2}$ is given in \eqref{eqI2}.
 Below, we give the discrete algorithm of the continuous algorithm
 as

\begin{algorithm}[h]
\label{Alg2}%
\caption{Newton dumped-Gauss-Seidel algorithm (Discrete version)}
\begin{algorithmic}[1]

\STATE
 \textbf{Input :}$\left\{ y_{0}^{\delta,h }, p_{0}^{\delta,h }, \varphi _{0}^{\delta,h }, \lambda _{0}^{\delta,h },\psi _{0}^{\delta,h },\delta ,\nu
,\omega_{y},\omega_{\varphi},\omega_{\psi},\varepsilon\right\} $
choose $\varphi _{0}^{\delta,h },\psi _{0}^{\delta,h }\in
\mathcal{W},\varepsilon $ and $\delta $ in $ \mathbb{R}_{+}^{\ast
}$;

\STATE\textbf{Begin:}\\

\STATE \textbf{Calculate } $J_{n-1} \leftarrow
J_{n-1}\left(y^{\delta,h}_{n-1}, \varphi^{\delta,h}_{n-1},
\psi^{\delta,h}_{n-1}\right)$

\STATE \textbf{If} $\left( A_{h }^{d }+\textbf{diag}(\beta
_{\delta }^{\prime }\left( y_{n-1}^{\delta,h }-\varphi
_{n-1}^{\delta,h }\right))+\textbf{diag}(\beta _{\delta }^{\prime
}\left(\psi _{n-1}^{\delta,h }- y_{n-1}^{\delta,h }\right))
\right) $ is singular \textbf{Stop.}

\STATE  \textbf{ Else} \STATE \textbf{ Solve} $\left( A_{h }^{d
}+\textbf{diag}(\beta _{\delta }^{\prime }\left( y_{n-1}^{\delta
,h}-\varphi _{n-1}^{\delta,h }\right))+\textbf{diag}(\beta
_{\delta }^{\prime }\left(\psi _{n-1}^{\delta,h }-
y_{n-1}^{\delta,h }\right) \right)) .r_{n}^{\delta } = -
\omega_{y}\left( A_{h }^{d } y_{n-1}^{\delta
,h}+\textbf{diag}(\beta _{\delta }\left( y_{n-1}^{\delta,h
}-\varphi _{n-1}^{\delta,h } \right))-\textbf{diag}(\beta _{\delta
}\left( \psi _{n-1}^{\delta,h } -y_{n-1}^{\delta,h }\right)) -f
\right)$ on $r_{n}^{\delta },$

\STATE \ \ \ \ \ \ \ \ \ \ \ \textbf{Then } $y_{n}^{\delta,h
}=y_{n-1}^{\delta,h }+$ $r_{n}^{\delta }.$ \\

 \STATE \textbf{End if}

\STATE \textbf{If} $\left( A_{h }^{d }+\textbf{diag}(\beta
_{\delta }^{\prime }\left( y_{n}^{\delta,h }-\varphi
_{n-1}^{\delta,h }\right)) +\textbf{diag}(\beta _{\delta }^{\prime
}\left( \psi _{n-1}^{\delta,h }
-y_{n}^{\delta,h }\right))\right)$ is singular \textbf{Stop.} \\

\STATE \ \ \ \ \ \ \ \ \  \textbf{\ Else} \\

\STATE \textbf{\ Solve } $\left( A_{h }^{d } +\textbf{diag}(\beta
_{\delta }^{\prime }\left( y_{n}^{\delta,h }-\varphi
_{n-1}^{\delta,h }\right))+\textbf{diag}(\beta _{\delta }^{\prime
}\left( \psi _{n-1}^{\delta ,h} -y_{n}^{\delta,h }\right)) \right)
p_{n}^{\delta,h }=y_{n}^{\delta,h }-z$
on $p_{n}^{\delta,h }.$\\

 \STATE \textbf{End if}
\STATE \textbf{Calculate} $\lambda_{n}^{\delta,h } = \nu A_{h }^{d
} \varphi _{n-1}^{\delta,h }+\textbf{diag}(\beta _{\delta
}^{\prime }\left( y_{n}^{\delta,h }-\varphi _{n-1}^{\delta,h
}\right)) p_{n}^{\delta,h } $

\STATE \textbf{If} $\left( \nu A_{h }^{d }+\textbf{diag}(\beta
_{\delta }^{\prime \prime }\left( \psi _{n-1}^{\delta,h
}- y_{n}^{\delta ,h}\right)) p_{n}^{\delta,h }\right) $ is not invertible \textbf{Stop.}\\

\STATE \ \ \ \ \ \ \ \ \  \textbf{Else} \\
\STATE \textbf{Solve } $\left(  \nu A_{h }^{d
}+\left(\textbf{diag}\left(\beta _{\delta }^{\prime \prime
}\left(\psi _{n-1}^{\delta,h } -y_{n}^{\delta ,h}\right)
\right)\right) p_{n}^{\delta,h }\right).r_{n}^{\delta }=$ \\
$-\omega_{\psi}\left(\nu A_{h}^{d}\psi _{n-1}^{\delta,h
}+\textbf{diag}(\beta _{\delta }^{\prime }\left( \psi
_{n-1}^{\delta,h }-y_{n}^{\delta,h }\right)) p_{n}^{\delta,h
}+\lambda_{n}^{\delta,h }\right)$ on $r_{n}^{\delta }.$

\STATE \ \ \ \ \ \ \ \ \ \ \ \textbf{Then }$ \psi _{n}^{\delta,h
}=\psi _{n-1}^{\delta,h }+$ $r_{n}^{\delta }$

\STATE \textbf{If} $\left( \nu A_{h }^{d }-\textbf{diag}(\beta
_{\delta }^{\prime \prime }\left(y_{n}^{\delta,h }- \varphi
_{n-1}^{\delta,h
}\right)) p_{n}^{\delta,h }\right) $ is not invertible \textbf{Stop.}\\

\STATE \ \ \ \ \ \ \ \ \  \textbf{Else} \\

\STATE \textbf{Solve }$\left(  \nu A_{h }^{d }-\textbf{diag}(\beta
_{\delta }^{\prime \prime }\left( y_{n}^{\delta,h }-\varphi
_{n-1}^{\delta,h } \right)) p_{n}^{\delta,h }\right).$
$r_{n}^{\delta }$=\\$-\omega_{\varphi}\left(\nu A_{h}^{d}\varphi
_{n-1}^{\delta,h }+\textbf{diag}(\beta _{\delta }^{\prime }\left(
y_{n}^{\delta,h }-\varphi _{n-1}^{\delta,h }\right))
p_{n}^{\delta,h }-\lambda_{n}^{\delta,h }\right)$ on
$r_{n}^{\delta }.$

\STATE \ \ \ \ \ \ \ \ \ \ \ \textbf{Then }$ \varphi
_{n}^{\delta,h }=\varphi _{n-1}^{\delta ,h}+$ $r_{n}^{\delta }$
\STATE \textbf{Calculate } $J_{n} \leftarrow
J_{n-1}\left(y^{\delta,h}_{n},
\varphi^{\delta,h}_{n},\psi^{\delta,h}_{n}\right)$

 \STATE \textbf{End if}

 \STATE \textbf{If } $|J_{n}-J_{n-1}| \leq  \varepsilon $ \textbf{Stop.}

\STATE \textbf{Ensure :} $s_{n}^{\delta }:=\left( y_{n}^{\delta
},\varphi_{n}^{\delta },\psi_{n}^{\delta },p_{n}^{\delta }\right)
$ \textbf{is a solution}

\STATE \ \ \ \ \ \ \ \ \  \textbf{Else; } $n\leftarrow n+1$,
\textbf{Go to} \textbf{Begin.}

 \STATE \textbf{End if}

\STATE \textbf{End algorithm}.
\end{algorithmic}
\end{algorithm}
\FloatBarrier

\begin{remark}
\label{Rem_2} Theorem \ref{thr2} is given for the continuous
problem and it is clear that for the discrete form of the proposed
algorithm, we must introduce the discretisation parameter $h$. But
for this discrete form of the algorithm $2$, it is very difficult
to give a sharp estimate of the Lipschitz constant $k$ given by
Theorem \ref{thr2}.
\end{remark}

\subsection{Numerical examples in one dimensional space}

In this section, we take $\Omega=[0,1]$ and we describe some
numerical experiments in one dimensional space based on the
previous algorithm. We also give some numerical tests when in each
test we vary one of the parameters $\omega$, $\delta$, $N$ and
$\nu$, where $f(x)= 100xcos(3 \pi x)$, $z(x)= cos(4 \pi x^2)$ and
$\nu>0$ are given. In the sequel, we note by $\epsilon_{n}$ the
quantity $\max \left\{
\parallel y_{n}-y_{n-1}\parallel _{\infty},\parallel \varphi
_{n}-\varphi _{n-1}\parallel _{\infty}\right\}$.

\subsubsection{Test 1: Study of the  dependence on the parameter
$\omega$ with $\protect\delta =h^{2}$, $\protect\nu =1$ and
$N=200$}
Numerical results are displayed in Table 1 according to the
variation of $\omega$. In Figure 1, we give the curves
corresponding to the controls $\varphi$ and $\psi$. Curves given
in Figure 2 show the contact region $I(y)$ between the state and
the control functions. Finally, Figure 3 gives graphical
variations in a log-log scale of $\epsilon_n$ and $J_{n}$ for each
iteration $n$.

\subsubsection{Test 2: Study of the  dependence on the parameter
$N$ with $\protect\delta =h^2$, $\protect\omega =0,75$ and
$\protect\nu =1$}
Numerical results are displayed in Table 2 according to the
variation of $N$. In Figure 4, curves corresponding to the
controls $\varphi$ and $\psi$ are shown. Curves given inFigure 5
show the contact region $I(y)$ between the state and the control
functions. Finally, Figure 6 gives graphical variations in a
log-log scale of $\epsilon_n$ and $J_{n}$ for each iteration $n$.

\subsubsection{Test 3: Study of the  dependence on the parameter $\protect%
\nu $ with $\protect\delta =h^2$, $\protect\omega =0,75$ and
$N=200$}

Numerical results are displayed in Table 3 according to the
variation of $\nu$. In Figure 7, curves corresponding to the
controls $\varphi$ and $\psi$ are shown. Curves given in Figure 8
show the contact region $I(y)$ between the state and the control
functions. Finally, Figure 9 gives graphical variations in a
log-log scale of $\epsilon_n$ and $J_{n}$ for each iteration $n$.

\subsubsection{Test 4: Study of the  dependence on the parameter $\protect%
\delta $ with $N=200$, $\protect\omega =0,75$ and $\protect\nu
=0.1$}
Numerical results are displayed in Table 4 according to the
variation of $\delta$. In Figure 10, curves corresponding to the
controls $\varphi$ and $\psi$ are shown. Curves given in Figure 11
show the contact region $I(y)$ between the state and the control
functions. Finally, Figure 12 gives graphical variations in a
log-log scale of $\epsilon_n$ and $J_{n}$ for each iteration $n$.

\subsection{Numerical examples in two dimensional space}
In this section, we describe some numerical experiments in two
dimensional space based on the previous algorithm. We also give
some numerical tests when in each test we vary one of the
parameters $\omega$, $\delta$, $N$ and $\nu$, where
$\Omega=[0,1]\times[0,1]$, $f(x,y)= x^3sin(2 \pi x^2)ycos(2 \pi
y^2)$ and $z(x,y)=sin(2 \pi x^2)cos(2\pi y^2)$) and
$\omega_{y}=\omega_{\varphi}=\omega_{\psi}=\omega$.
\subsubsection{Test 1: Study of the  dependence on the parameter
$\omega$ with $\protect\delta =h^4$, $\protect\nu =1$ and $N=40$}
Numerical results are displayed in Table 5 according to the
variation of $\omega$. Figure 13 gives graphical variations in a
log-log scale of $\epsilon_n$ and $J_{n}$ for each iteration $n$.
Curves given in Figure 14 and  Figure 15 corresponding to the
controls and state functions are shown.

\subsubsection{Test 2: Study of the  dependence on the parameter
$N$ with $\protect\delta =h^4$, $\protect\omega =0,5$ and
$\protect\nu =1$}

Numerical results are displayed in Table 6 according to the
variation of $N$. Figure 16 gives graphical variations in a
log-log scale of $\epsilon_n$ and $J_{n}$ for each iteration $n$.
Curves given in Figure 17 Figure 18 corresponding to the controls
and state functions are shown.

\subsubsection{Test 3: Study of the  dependence on the parameter $\protect%
\nu $ with $\protect\delta =h^4$, $\protect\omega =0,5$ and
$N=40$}

Numerical results are displayed in Table 7 according to the
variation of $\nu$. Figure 19 gives graphical variations in a
log-log scale of $\epsilon_n$ and $J_{n}$ for each iteration $n$.
Curves given in Figure 20 and Figure 21 corresponding to the
controls and state functions are shown.

\subsubsection{Test 4: Study of the  dependence on the parameter $\protect%
\delta $ with $N=40$, $\protect\omega =0,5$ and $\protect\nu =1$}
Numerical results are displayed in Table 8 according to the
variation of $\delta$. Figure 22 gives graphical variations in a
log-log scale of $\epsilon_n$ and $J_{n}$ for each iteration $n$.
Curves given in Figure 23 and Figure 24 corresponding to the
controls and state functions are shown.

\section{Conclusion and remarks}

We notice that techniques used in the paper of Ghanem et al.
\cite{Ghanemzireg} can be easily applied to the numerical resolution of the
problem considered in this work. The given numerical results are
acceptable although the convergence of the algorithm is not fast.
They also consolidate our perception given in Remarks \ref{Rem_1}
and \ref{Rem_2} about the Lipschitz constants. We can either apply other algorithms of
resolution (for example semismooth Newton methods ) \cite{KunischWachsmuth} or should improve the used algorithm by optimizing the
choice of the parameter (by the line search method, for example).

\end{document}